\documentclass[12pt,leqno,a4paper]{amsart}
\usepackage{amssymb}

\newcommand{\FF}{{\mathbb{F}}}
\newcommand{\ZZ}{{\mathbb{Z}}}

\newcommand{\bk}{{\mathbf{k}}}
\newcommand{\bB}{{\mathbf{B}}}
\newcommand{\bC}{{\mathbf{C}}}
\newcommand{\bG}{{\mathbf{G}}}
\newcommand{\bH}{{\mathbf{H}}}
\newcommand{\bL}{{\mathbf{L}}}
\newcommand{\bT}{{\mathbf{T}}}
\newcommand{\bU}{{\mathbf{U}}}
\newcommand{\bw}{{\mathbf{w}}}

\newcommand{\fA}{{\mathfrak{A}}}
\newcommand{\fS}{{\mathfrak{S}}}

\newcommand{\cE}{{\mathcal{E}}}
\newcommand{\cF}{{\mathcal{F}}}

\newcommand{\Irr}{{\operatorname{Irr}}}
\newcommand{\PGL}{{\operatorname{PGL}}}
\newcommand{\PSL}{{\operatorname{PSL}}}
\newcommand{\GL}{{\operatorname{GL}}}
\newcommand{\SL}{{\operatorname{SL}}}
\newcommand{\GU}{{\operatorname{GU}}}
\newcommand{\SU}{{\operatorname{SU}}}
\newcommand{\PSU}{{\operatorname{U}}}
\newcommand{\PSp}{{\operatorname{S}}}
\newcommand{\Sp}{{\operatorname{Sp}}}
\newcommand{\GO}{{\operatorname{GO}}}
\newcommand{\SO}{{\operatorname{SO}}}
\newcommand{\OO}{{\operatorname{O}}}
\newcommand{\Ca}{{\operatorname{(C1)}}}
\newcommand{\Cb}{{\operatorname{(C2)}}}

\newcommand{\tB}{{\tilde B}}
\newcommand{\tD}{{\tilde D}}
\newcommand{\tG}{{\tilde G}}

\newcommand{\tw}[1]{{}^{#1}\!}
\def\pmod#1{~({\rm mod}~#1)}
\newcommand{\flr}[1]{{\lfloor #1\rfloor}}

\let\al=\alpha
\let\bt=\beta
\let\ga=\gamma
\let\eps=\epsilon
\let\del=\delta
\let\la=\lambda
\let\mdash=\models

\newtheorem{thm}{Theorem}[section]
\newtheorem{lem}[thm]{Lemma}
\newtheorem{prop}[thm]{Proposition}
\newtheorem{cor}[thm]{Corollary}

\newtheorem*{thmA}{Theorem 1}
\newtheorem*{conjA}{Conjecture}

\theoremstyle{remark}
\newtheorem{rem}[thm]{Remark}
\newtheorem{exmp}[thm]{Example}

\begin{document}

\title[Number of characters in blocks]{On the number of characters\\ in blocks of quasi-simple groups}

\date{\today}

\author{Gunter Malle}
\address{FB Mathematik, TU Kaiserslautern, Postfach 3049,
         67653 Kaiserslautern, Germany.}
\email{malle@mathematik.uni-kl.de}

\thanks{This paper is based upon work supported by the National Science
Foundation under Grant No. DMS-1440140 while the author was in residence at
the Mathematical Sciences Research Institute in Berkeley, California,
during the Spring 2018 semester.
The author also gratefully acknowledges financial support by SFB TRR 195.}



\begin{abstract}
We prove, for primes $p\ge5$, two inequalities between the fundamental
invariants of Brauer $p$-blocks of finite quasi-simple groups: the number of
characters in the block,
the number of modular characters, the number of height zero characters, and
the number of conjugacy classes of a defect group and of its derived subgroup.
For this, we determine these invariants explicitly, or at least give
bounds for them for several classes of classical groups.
\end{abstract}

\maketitle


\section{Introduction} \label{sec:intro}

In this paper we study two inequalities for the number of simple modules
in blocks of finite groups that had been proposed in joint work with
G.~Navarro \cite{MN06}.

Let $G$ be a finite group, $p$ be a prime, and $B$ a Brauer $p$-block of $G$
with defect group $D$. We write $D'=[D,D]$ for the derived subgroup of $D$.
Let $k(B)$ denote the number of ordinary irreducible characters in $B$ and
$k_0(B)$ the number of irreducible characters in $B$ of height zero (see
Section~\ref{sec:red}).
We write $k(D)$ for the number of conjugacy classes of $D$, and $l(B)$ for the
number of simple modules in characteristic~$p$ of $B$. In Malle--Navarro
\cite{MN06} we proposed the following inequalities connecting
these invariants:

\begin{conjA}
 Let $B$ be a $p$-block of a finite group with defect group $D$. Then
 $$k(B)\,/\,k_0(B)\le k(D')\quad\text{and}\eqno{\Ca}$$
 $$k(B)\,/\,l(B)\le k(D).\qquad\eqno{\Cb}$$
\end{conjA}

It was shown in \cite{MN06}
that these statements are satisfied, for example, for blocks with a normal
defect group, but also for blocks of symmetric groups.
In the present paper we investigate a possible minimal counterexample to $\Ca$
or~$\Cb$ for blocks of non-abelian simple groups and their covering groups. Our
main result is:

\begin{thmA}
 Let $B$ be a $p$-block of a finite quasi-simple group $G$. Assume one of the
 following holds:
 \begin{enumerate}
  \item[\rm(1)] $p\ge5$, or
  \item[\rm(2)] $G$ is a covering group of an alternating group, of a sporadic
    group or of a simple group of Lie type in characteristic~$p$.
  \end{enumerate}
 Then $B$ is neither a minimal counterexample to~$\Ca$ nor to~$\Cb$.
\end{thmA}

Here, we say that $(G,B)$ is a minimal counterexample to (C$i$) if (C$i$) holds
for all $p$-blocks $B_1$ of groups $G_1$ with $|G_1/Z(G_1)|$ strictly smaller
than $|G/Z(G)|$ and with defect groups isomorphic to those of $B$.

While we obtain partial results for the primes $p=2,3$, the case of groups
of Lie type in non-defining characteristic seems out of reach at the moment;
the case $p=3$ might be accessible to a similar but more tedious
investigation, but the prime $p=2$ will require a different approach.

Let us remark that no reduction of our conjecture to the case of
(quasi-)simple groups has been proposed so far.

The conjectured inequalities are closely related to three long-standing
conjectures in representation theory. Brauer's $k(B)$-conjecture claims that
$k(B)\le |D|$, while the Alperin--McKay conjecture relates the number $k_0(B)$
to the analogous number for the Brauer correspondent block of the normaliser
$N_G(D)$ of a defect group $D$ of $B$. Finally, Brauer's height zero conjecture
proposes that $k(B)=k_0(B)$ if and only if $D$ is abelian. In fact, we show in
Theorem~\ref{thm:abelian} that the statement of the known direction of this
conjecture implies~$\Ca$ in the case of abelian defect groups.
Let us point out one motivation for studying these: if the celebrated  
Alperin--McKay conjecture holds true, then $k_0(B)\le|D/D'|$ by the proven
kGV-conjecture. But then $\Ca$ claims that $k(B)\le |D| (k(D')/|D'|)$, which
in general is much smaller than the bound $|D|$ stipulated by Brauer's
$k(B)$-conjecture. Thus, if true, our conjecture would yield a better bound on
$k(B)$ than Brauer's (yet unproven) one.
\medskip

The paper is built up as follows: In Section~\ref{sec:red} we present some
first reductions. The covering groups of alternating groups are treated in
Section~\ref{sec:alt}. In Section~\ref{sec:def char} we verify the inequalities
for blocks of quasi-simple groups of Lie type for the defining prime in
Corollary~\ref{cor:def char}.
In Section~\ref{sec:cross class} we first present some general results for
groups of Lie type in cross characteristic and then show the conjecture for
blocks of groups of classical type (Corollary~\ref{cor:class}).
To this end we also
derive explicit formulas for invariants of unipotent blocks which we believe to
be of independent interest (see Theorems~\ref{thm:k0 SL} and~\ref{thm:k0 BCD}
and Proposition~\ref{prop:k0 PGL}). The groups of exceptional type are then
considered in Section~\ref{sec:cross exc}, see Theorem~\ref{thm:exc}.
\medskip

\noindent{\bf Acknowledgement:} I thank Frank Himstedt for providing me
with information on the principal 2- and 3-blocks of $\tw3D_4(q)$ and Donna
Testerman for a careful reading of a preliminary version.

\section{First reductions} \label{sec:red}

Let $G$ be a finite group and $p$ a prime. Let $\bk$ be an algebraically closed
field of characteristic~$p$. The decomposition of the group algebra $\bk G$
into a sum of minimal 2-sided ideals (called Brauer $p$-blocks) induces a
corresponding subdivision of the set of isomorphism classes of irreducible
$\bk G$-modules, as well as of the set $\Irr(G)$ of irreducible complex
characters of~$G$. If $B$ is such a $p$-block of $G$, then we write $l(B)$ for
the number of
isomorphism classes of irreducible $\bk G$-modules belonging to $B$, and
$k(B)$ for the number of irreducible complex characters in $B$. Associated to
a block $B$ is a conjugacy class of $p$-subgroups of $G$, the so-called defect
groups. If $D$ is a defect group of $B$, then $\bk D$ is a single $p$-block,
and so $k(\bk D)$ coincides with the number of conjugacy classes of $D$, which
we denote by $k(D)$. We write $k_0(B)$ for the number of characters in $B$
\emph{of height~0}, that is, for the number of elements in
$$\{\chi\in\Irr(B)\mid \chi(1)|D|/|G|_p\not\equiv0\pmod p\}.$$
We denote by $H'=[H,H]$ the derived subgroup of a group $H$.

We start with the following consequence of the proven direction of Brauer's
height zero conjecture:

\begin{thm}   \label{thm:abelian}
 Let $B$ be a $p$-block of a finite quasi-simple group with abelian defect
 groups. Then $B$ is not a counterexample to inequalities~$\Ca$ and~$\Cb$.
\end{thm}

\begin{proof}
Let $D$ denote a defect group of the block $B$. If $D$ is abelian, then
$k(D')=1$, and by the known direction of Brauer's height zero conjecture
\cite{KM13} we have $k_0(B)=k(B)$. So $\Ca$ holds trivially (with equality).
As pointed out by Sambale \cite[p.~22]{Sa14}, $\Cb$ holds by a result of Feit.
\end{proof}


Sambale has proved the validity of our inequalities for several types of
defect groups; we will need the following cases:

\begin{prop}   \label{prop:Sambale}
 Let $B$ be a 2-block of a finite group whose defect group is either
 metacyclic or a central product of a metacyclic group with a cyclic group.
 Then~$\Ca$ and~$\Cb$ are satisfied for $B$.
\end{prop}

\begin{proof}
The claim for metacyclic groups is proved in \cite[Cor.~8.2]{Sa14}, the one for
central products in \cite[Thm.~9.1]{Sa14}.
\end{proof}

The following bound (see Pantea \cite[Prop.~3.2]{P04}) on the number of
conjugacy classes of $p$-groups will be useful in dealing with small cases:

\begin{prop}   \label{prop:p-group}
 Let $D$ be a $p$-group of order $|D|=p^n$. Then $k(D)\ge p^2+(n-2)(p-1)$.
\end{prop}

\begin{prop}   \label{prop:spor}
 Let $G$ be a covering group of a sporadic simple group or of $\tw2F_4(2)'$.
 Then inequalities~$\Ca$ and~$\Cb$ hold for all blocks of $G$.
\end{prop}

\begin{proof}
This is an elementary check based on the known character tables of the
quasi-simple groups in question, using the bound from
Proposition~\ref{prop:p-group}. Only the first inequality is not immediate in
all cases, more precisely, not so for certain faithful 2-blocks of
$$4.M_{22},\ 2.HS,\ 2.Ru,\ 2.Suz,\ 6.Suz,\ 2.Fi_{22},\ 6.Fi_{22},\
  Fi_{23} \text{ and }2.Co_1.$$
In these cases it suffices to check that defect groups $D$ satisfy $k(D')\ge8$.
\end{proof}

\begin{prop}   \label{prop:exc cover}
 Let $G$ be an exceptional covering group of a simple group of Lie type or of
 the alternating group $\fA_7$. Then inequalities~$\Ca$ and~$\Cb$ hold for all
 blocks of $G$.
\end{prop}

\begin{proof}
Again, the ordinary character tables of all groups in question are known
and the claim can be checked directly. The first inequality is not immediate
for certain faithful 2-blocks of
$$12_1.\PSU_4(3),\ 12_2.\PSU_4(3),\ 2.\PSU_6(2),\ 6.\PSU_6(2),\ 2.\OO_8^+(2),\
  2.F_4(2),\ 2.\tw2E_6(2) \text{ and }6.\tw2E_6(2),$$
but again in these cases it suffices to verify that $k(D')\ge8$.
\end{proof}

\section{Alternating groups}   \label{sec:alt}

In \cite[Prop.~4.4 and~4.7]{MN06} we proved that all blocks of symmetric
groups $\fS_n$ satisfy our inequalities. The corresponding statement for
blocks of the alternating groups $\fA_n$ and of the 2-fold covering groups
$2.\fA_n$, which we will derive here, is not an immediate consequence of this
latter result, though. Our proofs will crucially rely on various results of
Olsson.

Recall that $p$-blocks of $\fS_n$ are parametrised by $p$-cores, and that
their block theoretic invariants only depend on their weight.
Let $\tilde B=\tilde B(w)$ be a $p$-block of $\fS_n$ of weight~$w$. Then
$k(p,w):=k(\tilde B(w))$ is the number of $p$-multipartitions of $w$.
Let $P_i$ denote the $i$-fold wreath product of the cyclic group of order $p$,
that is, $P_i=C_p\wr\cdots\wr C_p$ (with $i$ terms). Then a defect group of
$\tilde B(w)$ is a direct product
$\tilde D=\tilde D(w)=\prod_{i=0}^r P_{i+1}^{a_i}$, where
$w=\sum_{i=0}^r a_ip^i$, with $0\le a_i<p$, is the $p$-adic decomposition
of~$w$ (see \cite[Prop.~11.3]{Ol93}).
In particular, defect groups are abelian when the weight $w$ is less than $p$.

\begin{prop}   \label{prop:An p>2}
 Let $G=2.\fA_n$ with $n\ge5$, and $p>2$ an odd prime. Then all $p$-blocks of
 $G$ satisfy inequalities~$\Ca$ and~$\Cb$.
\end{prop}

\begin{proof}
We start with some observations on $p$-blocks of $\fS_n$. Let $\tilde B$ be
a $p$-block of $\fS_n$, of weight $w$, and let $D\le \fS_n$ be a defect
group of $\tilde B$. According to \cite[Prop.~4.7]{MN06} for $p\ge5$ we have
$k(\tilde B)\le k(D')$, and in fact the given bound shows that also
$k(\tilde B)\le k(D)$. Now assume that $p=3$. Here we still have
$k(\tilde B)=k(p,w)\le k(D')$ whenever $w\ge 18$ and $k(p,w)\le k(D)$ when
$w\ge4$. The cases in which these inequalities are not satisfied are collected
in Table~\ref{tab:p=3}. An entry ``--'' signifies that the quotient is at
most~1 and hence our inequality holds.

\begin{table}[htb]
\caption{3-blocks of small weight}   \label{tab:p=3}
$\begin{array}{r|cccccccccccc}
          w& 3& 4& 5& 6& 7& 8& 9& 10& 11& 13& 14& 17\\
\noalign{\hrule}
 k(\tilde B)/k(D')\le& 3& 6& 12& 3& 6& 10& 2& 3& 5& 2& 3& 2\\
        k_0(\tilde B)& 9& 27& 81& 54& 162& 486& 27& 81& 243& 648& 2187& 13122\\
      k(\hat B)/k(D')& 2& 3& 4& -& 2& 2& -& -& -& -& -& - \\
\noalign{\hrule}
  k(\tilde B)/k(D)\le& 2& -& -& -& -& -& -& -& -& -& -& -\\
\noalign{\hrule}
  \end{array}$
\end{table}

Now let $B$ be a $p$-block of $\fA_n$, and $\tilde B$ a $p$-block of $\fS_n$
covering $B$. As $p$ is odd, any defect group $D$ of $B$ is also a defect group
of $\tilde B$. Now by \cite[Prop.~4.10]{Ol90} we have $k(B)\le k(\tilde B)$, so
by what we showed above
$$k(B)/k_0(B)\le k(B)\le k(\tilde B)\le k(D')\quad\text{and}\quad
  k(B)/l(B)\le k(B)\le k(\tilde B)\le k(D),$$
whence $\Ca$ and~$\Cb$ hold for $B$, unless $p=3$ and
$w\le17$. Here, we certainly always have $k_0(B)\ge k_0(\tilde B)/2$, and
$k_0(\tilde B)=\prod_{i=0}^r k(3^{i+1},a_i)$ by \cite[Prop.~12.4]{Ol93}.
The relevant values are given in Table~\ref{tab:p=3} from which $\Ca$ follows.
Similarly we have $l(B)\ge l(\tilde B)/2$, and $l(\tilde B)=k(2,w)$ by
\cite[Prop.~12.8]{Ol93} which is greater than~2 for $w=3$, so $\Cb$ is also
satisfied.
\par
Next, let $\hat B$ be a faithful $p$-block of $2.\fA_n$. As pointed out in
\cite[Rem.~13.18 and Prop.~3.19]{Ol93} for any spin block of $2.\fA_n$ there
is some $m\ge1$ and a spin block of $2.\fS_m$ having the same invariants, so we
may assume that $\hat B$ is in fact a faithful block of $2.\fS_n$. Let
$\tilde B$ denote a block of some symmetric group $\fS_m$ with the same weight
as $\hat B$ and (hence) isomorphic defect groups. According to
\cite[Prop.~13.14]{Ol93} we again have $k(\hat B)\le k(\tilde B)$. We can thus
argue
as before unless $p=3$ and $w\le17$. When $p=3$ by \cite[Cor.~13.6]{Ol93} we
have that $k(\hat B)=\frac{3}{2}\sum_{i=0}^wq(i)p(w-i)$, where $q(i)$ is the
number of strict partitions of $i$ (see \cite[Prop.~9.6(i)]{Ol93}). Then
$k(\hat B)/k(D')\le1$ whenever $w\ge9$, and the exceptions are again listed in
Table~\ref{tab:p=3}. It is straightforward to check that the remaining five
weights do not lead to a counterexample to~$\Ca$. By
\cite[Prop.~13.17]{Ol93}, for all $w$ we have $l(\hat B)\ge k(1,w)$, the number
of partitions of $w$. Visibly this is larger than $k(\tilde B)/k(D)$ for $w\ge3$
thus showing inequality~$\Cb$.
\end{proof}

\begin{rem}
The proof shows that $p$-blocks $B$ of covering groups of alternating groups
for $p\ge3$ always satisfy the strengthened form $k(B)\le k(D)$ of Brauer's
$k(B)$-conjecture, which trivially implies at least inequality~$\Cb$, unless
$p=w=3$.
\end{rem}

\begin{prop}   \label{prop:An p=2}
 All $2$-blocks of $\fA_n$ with $n\ge5$ satisfy inequalities~$\Ca$ and~$\Cb$.
\end{prop}

\begin{proof}
We first consider a block $\tilde B$ of $\fS_n$ of weight $w\ge3$ and with
defect group $\tilde D$. Then by \cite[Prop.~11.4]{Ol93} we have
$k(\tilde B)=k(2,w)$, while $k(\tilde D)=\prod_{i=0}^r k(\tilde D(2^i))^{a_i}$,
where $w=\sum_i a_i 2^i$ is the 2-adic decomposition of $w$. Then using the
estimates in \cite[Lemma~4.3]{MN06} when $w$ is large, and explicit values for
small $w$, it is readily seen that
$$k(\tilde B)\le k(\tilde D)\qquad\text{ for all $w$}.$$
Furthermore, again with \cite[Lemma~4.3]{MN06} we get that
$$k(\tilde B)\le 4\,k(\tilde D')\qquad\text{ when $w\ne3,7$}.$$
In particular, $\Cb$ is satisfied for all 2-blocks of $\fS_n$, and~$\Ca$ holds at
least when $w\ne3,7$ as then $k_0(\tilde B)\ge4$. For $w=3,7$, $\Ca$ follows by
using the explicit values of $k_0(\tilde B)$ (see also
\cite[Prop.~4.4 and~4.7]{MN06}).
\par
Now let $B$ be a 2-block of $\fA_n$, and $\tilde B$ the 2-block of $\fS_n$
covering $B$. Let $D$ be a defect group of $B$ and $\tilde D\ge D$ a defect
group of $\tilde B$. We may assume that $\tilde B$ has weight at least~3, as
otherwise $|\tilde D|\le8$ and hence $D=\tilde D\cap\fA_n$ is abelian.
According to \cite[Prop.~4.13]{Ol90} we always have $k(B)\le k(\tilde B)$.
\par
If $w$ is odd, then in fact $k(B)=k(\tilde B)/2$ by \cite[Prop.~4.5]{Ol90}. As
$|\tilde D:D|=2$ this implies inequality~$\Cb$, from the fact shown
above that $k(\tilde B)\le k(\tilde D)$. For inequality~$\Ca$, by
\cite[Prop.~4.6]{Ol90} we have
$k_0(B)=k_0(\tilde B)/2$, while $D'=\tilde D'$ by \cite[Prop.~4.15]{Ol90}, so
$$k(B)/k_0(B)=k(\tilde B)/k_0(\tilde B)\le k(\tilde D')=k(D')$$
by the result for the block $\tilde B$ of $\fS_n$.
\par
Now assume that $w$ is even. Then $l(B)\ge 2$ whenever $D$ is nonabelian,
which gives
$$k(B)/l(B)\le k(\tilde B)/2\le k(\tilde D)/2\le |\tilde D:D|k(D)/2=k(D),$$
whence~$\Cb$. As for~$\Ca$, we still have
$$k(B)\le k(\tilde B)\le 4k(\tilde D')\le 8k(D')\le k_0(B)k(D')$$
for $w\ne3,7$, as $k_0(B)\ge8$ for $w\ge4$. It remains to check the two cases
$w=3,7$, which is straightforward.
\end{proof}

\begin{prop}   \label{prop:2.An p=2}
 All $2$-blocks of $2.\fA_n$ with $n\ge5$ satisfy inequalities~$\Ca$ and~$\Cb$.
\end{prop}

\begin{proof}
Let $\hat B$ be a 2-block of $G=2.\fA_n$ and $B$ the block
of $\fA_n$ contained in $\hat B$. If $\hat D$ is a defect group of $\hat B$,
then $D=\hat D/Z$ is a defect group of $B$, where $Z=Z(G)$. In particular,
$k(\hat D)\ge k(D)$ and $k(\hat D')\ge k(D')$. Let $\tilde B$ denote the
block of $\fS_n$ covering $B$. By \cite[Thm.~C]{Na17} we have
$k(\hat B)\le2k(B)$. So using the results for blocks of $\fS_n$ shown in the
proof of Proposition~\ref{prop:An p=2} we get
$$k(\hat B)\le 2k(B)\le 2k(\tilde B)\le 2k(\tilde D)\le 4k(D)\le 4k(\hat D)
  \le l(\hat B)k(\hat D)$$
as $l(\hat B)=l(B)\ge4$ for $w\ge 4$. Thus we obtain~$\Cb$. For $w=3$ we have
$k(\hat B)=9$, $l(\hat B)=3$ and $k(\hat D)\ge4$. Also, again using the result
for $\fS_n$,
$$k(\hat B)\le 2k(B)\le 2k(\tilde B)\le 8k(\tilde D')\le 16 k(D')
  \le 16k(\hat D')$$
for $w\ne 3,7$. As $k_0(\hat B)\ge16$ for $w\ge6$, this proves~$\Ca$ when
$w\ge 8$. The cases of small $w$ can again be checked individually.
\end{proof}

Arguing along the same lines it is straightforward to show that the 2-blocks of
$2.\fS_n$ also satisfy~$\Ca$ and~$\Cb$.

\begin{thm}   \label{thm:alt}
 Let $G$ be a covering group of $\fA_n$, $n\ge5$. Then all $p$-blocks of $G$
 satisfy inequalities~$\Ca$ and~$\Cb$ for all primes $p$.
\end{thm}

\begin{proof}
The blocks of the exceptional covering groups of $\fA_6\cong\PSL_2(9)$ and
$\fA_7$ have been
considered in Proposition~\ref{prop:exc cover} while the blocks of $\fA_n$
have been dealt with in Propositions~\ref{prop:An p>2} and~\ref{prop:An p=2}.
Finally, the faithful blocks of $2.\fA_n$ for odd primes were again handled in
Proposition~\ref{prop:An p>2} and the 2-blocks in
Proposition~\ref{prop:2.An p=2}. This completes the proof.
\end{proof}

\section{Groups of Lie type in their defining characteristic} \label{sec:def char}
In this section we verify the inequalities~$\Ca$ and~$\Cb$ for blocks of
quasi-simple groups of Lie type in their defining characteristic. Partial
results had already been obtained in \cite[Prop.~3.2]{MN06}. In particular,
both inequalities were shown to hold for groups obtained from simple algebraic
groups of adjoint type. Thus, neither the $p$-blocks of Suzuki and Ree groups
nor those of groups of type $G_2$, $\tw3D_4$, $F_4$ or $E_8$ do yield
counterexamples. We can hence discard them from our discussion here.

Let $\bG$ be a simple algebraic group of simply connected type over an
algebraic closure of the finite field $\FF_p$ and $F:\bG\rightarrow\bG$ a
Frobenius endomorphism with respect to an $\FF_q$-rational structure, where
$q=p^f$. Let
$G=\bG^F$ be the finite group of fixed points. Let $\bT\le\bB\le\bG$ be an
$F$-stable maximal torus in an $F$-stable Borel subgroup of $\bG$, and
$\bU=R_u(\bB)$ the unipotent radical of $\bB$. Let $\Phi$ be the root system
of $\bG$ with respect to $\bT$ and $\Phi^+\subset\Phi$ the positive system
defined by $\bB$, with base $\Delta\subset\Phi^+$. We write $r:=|\Delta|$ for
the rank of the algebraic group $\bG$. For $\al\in\Phi^+$ let $\bU_\al\le\bU$
denote the corresponding root subgroup.
Set $U:=\bU^F$, a Sylow $p$-subgroup of $G$.

\begin{lem}   \label{lem:k(U)}
 We have $k(U)\ge q^{|\Delta|}$ and $k(U')\ge q^{2|\Delta|-3}$.
\end{lem}

\begin{proof}
According to Chevalley's commutator formula (see e.g.~\cite[Thm.~11.8]{MT}),
$\bU'$ is contained in the subgroup
$\bU_1:=\langle\bU_\al\mid \al\in\Phi^+\setminus\Delta\rangle$,
and $\bU/\bU_1\cong\prod_{\al\in\Delta}\bU_\al$. Clearly $\bU_1$ is also
$F$-stable, and we set $U_1=\bU_1^F$. But then
$$U/U_1=\bU^F/\bU_1^F\cong(\bU/\bU_1)^F\cong(\prod_{\al\in\Delta}\bU_\al)^F$$
has order $q^{|\Delta|}$, see \cite[Cor.~23.9]{MT}. Now clearly
$k(U)\ge k(U/U')\ge k(U/U_1)$, which proves the first claim.  \par
For the second claim note that again by the commutator formula $U''$ is
contained in the subgroup
$$\bU_2:=\langle \bU_{\al+\beta}\mid \al,\beta\in\Phi^+\setminus\Delta\rangle$$
(where we let $\bU_{\al+\beta}:=1$ if $\al+\bt\notin\Phi$). Now
$\bU_1/\bU_2\cong\prod_\ga\bU_\ga$ where $\gamma$ runs over
$$\{\ga\in\Phi^+\setminus\Delta\mid \ga\ne\al+\bt\text{ with }
  \al,\bt\in\Phi^+\setminus\Delta\},$$
that is, $\bU_1/\bU_2$ contains a subgroup isomorphic to the product of the
root subgroups for roots which are the sum of two or three simple roots.
By assumption $\bG$ is simple, so its root system $\Phi$ is indecomposable.
It is easily seen that any indecomposable root system of rank~$r$ has at least
$2r-3$ such roots. Now first assume that $U'=U_1$. Then the preceding argument
shows that $|U'/U''|\ge|\bU_1^F/\bU_2^F|=|(\bU_1/\bU_2)^F|\ge q^{2|\Delta|-3}$.
\par
Finally consider the case that $U'<U_1$. Then by \cite[Lemma~7]{H73} we have
$G$ is of type $B_n(2)$, $F_4(2)$, $G_2(2)$ or $G_2(3)$. In the first case,
$U'$ has index~2 in $U_1$ by \cite[Prop.~3]{C11}, but in this case
$|(\bU_1/\bU_2)^F|\ge q^{2|\Delta|-1}$, so we conclude as before. The other
three cases can be checked by direct computation.
\end{proof}

An asymptotic version of the following result for principal blocks had
already been obtained in \cite[Thm.~3.1]{MN06}:

\begin{thm}   \label{thm:def char}
 Let $G=\bG^F$ be quasi-simple of Lie type in characteristic~$p$. Then the
 $p$-blocks of $G$ satisfy inequalities~$\Ca$ and~$\Cb$.
\end{thm}

\begin{proof}
Let $B$ be a $p$-block of $G$. Then by a result of Humphreys (see
\cite[Thm.~6.18]{CE}) either $B$ is of defect zero and $B$ contains just the
Steinberg character of $G$ (whence our claim holds trivially), or it is of full
defect. So in the latter case, the maximal unipotent subgroup $U$ is a defect
group of $B$.  \par
Let's first consider~$\Ca$.
According to \cite[Thm.~1.1]{FG12} we have $k(B)\le k(G)\le 27.2 q^r$, where
$r=|\Delta|$ is the rank of $\bG$. Moreover, $k(U')\ge q^{2r-3}$ by
Lemma~\ref{lem:k(U)}. The number of $p'$-characters of $G$ is bounded below
by the number of semisimple characters (in the sense of Lusztig), hence of
semisimple conjugacy classes of the dual group $G^*$. Observe that $\bG^*$ is
simple if $\bG$ is. The number of semisimple conjugacy classes in the fixed
points of simple algebraic groups under Frobenius maps was determined in
\cite[Thm.~4.1(b) and Table~2]{BL13}. Let $d:=|Z(G)|$.  Then the description
in \cite[\S4.2]{BL13} shows that any of the sets $\Irr(G|\theta)$, for
$\theta\in\Irr(Z(G))$, contains at least $(q^r-1)/d$ irreducible characters
of $p'$-degree. Observe that $d\le \min\{r+1,q+1\}$ in all cases. So we are
done if we can show that
$$27.2\,d\le(q^r-1)q^{r-3}.$$
This holds whenever $r\ge4$. Moreover for $r=3$ it holds whenever $q>3$. For
the finitely many groups of rank~3 with $q\in\{2,3\}$ the assertion can be
checked from their
known character tables. On the other hand, for $r=1$ we have $G=\SL_2(q)$
which has abelian Sylow $p$-subgroups, whence the claim holds by
Theorem~\ref{thm:abelian}. Thus we are left with the case that $r=2$. Then
$\bG$ is either of type $A_2$, $C_2$ or $G_2$. We had already dealt with the
case $G_2$ in \cite[Prop.~3.2]{MN06}. For the other two types of groups, better
bounds on $k(G)$ are available,
namely $k(G)\le q^r+Aq^{r-1}$ for a certain constant $A$ which is at most~12
when $q\ge5$ \cite[Prop.~3.6, 3.10, Thm.~3.12, 3.13]{FG12}. It is easily
seen that this rules out the case $q\ge5$. For the finitely many groups with
$q\le4$ the claim can be checked from their known character tables.
\par
For inequality~$\Cb$ note that since $\bG$ is of simply connected type we have
that $G$ has precisely $l(G)=q^r$ semisimple, that is, $p'$-conjugacy classes
(see e.g. \cite[Thm.~4.1]{BL13}). So there are $l(G)-1=q^r-1$ simple
$\bk G$-modules in blocks of full defect. Let $d:=|Z(G)|$. The description
in \cite[\S4.2]{BL13} shows that any of the $d$-blocks $B$ of full defect
contains exactly $l(B)=(q^r-1)/d$ modular irreducibles. Since $k(U)\ge q^r$ by
Lemma~\ref{lem:k(U)}, here it suffices to show that
$$27.2\,d\le q^r-1.$$
As before, we have $d\le\min\{r+1,q+1\}$. Then the above inequality holds for
all $r\ge7$, and for $q\ge 8-r$ for $r\ge3$. The character tables of all
remaining groups are available in
GAP \cite{GAP}. In rank~2, we can again replace the bound for $k(G)$ by the
smaller values cited above to conclude unless $q\le3$. These last few groups
can again be checked individually.
\end{proof}

\begin{cor}   \label{cor:def char}
 Let $H$ be a quasi-simple group of Lie type in characteristic~$p$. Then all
 $p$-blocks of $H$ satisfy inequalities~$\Ca$ and~$\Cb$.
\end{cor}

\begin{proof}
By Proposition~\ref{prop:exc cover} we may assume that $H$ is not an
exceptional covering group of the simple group $H/Z(H)$. But then $|Z(H)|$ is
prime to $p$, so any $p$-block of $H$ is a $p$-block of a quasi-simple group
$G=\bG^F$ as in Theorem~\ref{thm:def char}, and the claim follows.
\end{proof}

\section{Groups of classical Lie type in non-defining characteristic}   \label{sec:cross class}
We now turn to $\ell$-blocks of groups of Lie type in characteristic~$p$,
where $\ell$ is different from $p$. We keep the algebraic group setup from
Section~\ref{sec:def char} and let $G=\bG^F$ for $\bG$ simple of simply
connected type defined over $\FF_q$.

\subsection{On non-unipotent blocks}

Let $B$ be an $\ell$-block of $G$. Then by a fundamental result of Brou\'e
and Michel (see e.g. \cite[Thm.~9.12(i)]{CE}) there is a semisimple
$\ell'$-element $s\in G^*:=\bG^{*F}$, where
$\bG^*$ is dual to $\bG$ with a Steinberg map also denoted $F$, such that
$$\Irr(B)\subseteq\cE_\ell(G,s):=\coprod_t\cE(G,st)$$
with the union running over $\ell$-elements $t\in C_{G^*}(s)$ up to conjugation.

\begin{lem}   \label{lem:BDR}
 Let $B$ be an $\ell$-block of $G$ such that $\Irr(B)\subseteq\cE_\ell(G,s)$.
 Assume that $C_{\bG^*}^\circ(s)$ lies in a proper $F$-stable Levi subgroup
 of $\bG^*$. Then $B$ is not a minimal counterexample to the inequalities~$\Ca$
 and~$\Cb$.
\end{lem}

\begin{proof}
Let $\bL^*<\bG^*$ be a proper $F$-stable Levi subgroup of $\bG^*$ containing
$C_{\bG^*}^\circ(s)$, and let $\bL\le\bG$ be dual to $\bL^*$, with $F$-fixed
points $L=\bL^F$.
By the main result of \cite{BDR17} any block $B$ in $\cE_\ell(G,s)$ is Morita
equivalent to a block $b$ of $L$ in $\cE_\ell(L,s)$ with isomorphic defect
groups. Thus all block theoretic invariants occurring in our inequalities
agree for $B$ and $b$, and as $L$ is proper in $G$ by assumption, $B$
cannot be a minimal counterexample.
\end{proof}

Elements $s\in G^*$ such that $C_{\bG^*}^\circ(s)$ does not lie in a proper
$F$-stable Levi subgroup of $\bG^*$ are called \emph{isolated}. So as far as
minimal counterexamples are concerned we only need to consider blocks in
isolated series. One important case is for $s=1$, that is, for unipotent
blocks.

\begin{prop}   \label{prop:Enguehard}
 Let $B$ be an isolated, non-unipotent $\ell$-block of a quasi-simple group
 of Lie type $H$ for a prime $\ell\ge5$ that is good for $H$. Then $B$ is not
 a minimal counterexample to either~$\Ca$ or~$\Cb$.
\end{prop}

\begin{proof}
By Proposition~\ref{prop:spor} we may assume that $H$ is not an exceptional
covering group. Thus, $H=G/Z$, where $G=\bG^F$ is as above and $Z\le Z(G)$. By
Lemma~\ref{lem:BDR} we may assume that $\bG$ is not of type $A$ as the
only isolated element in type $A$ is the identity, which corresponds to the
unipotent blocks. But then $\ell$ good implies that $\ell$ does not divide
$|Z(G)|$, so we may consider $B$ as an $\ell$-block of $G$. By the main result
of Enguehard \cite[Thm.~1.6]{En08} there is a group $G_1$, with $|G_1/Z(G_1)|$
strictly smaller than $|G/Z(G)|$ (since $B$ is not unipotent) with an
$\ell$-block $B_1$ having the same invariants ($k(B)$, $l(B)$ and defect group)
as $B$ and with a height preserving bijection $\Irr(B)\rightarrow\Irr(B_1)$.
In particular, $B$ and any block of $H$ dominated by $B$ is not a minimal
counterexample to~$\Ca$ or~$\Cb$.
\end{proof}

We note one further reduction which will be used for isolated 5-blocks of
$E_8(q)$:

\begin{lem}   \label{lem:one block}
 Let $s\in G^*$ be a non-central semisimple $\ell'$-element with connected
 centraliser $\bC^*=C_{\bG^*}(s)$ and let $\bC$ be dual to $\bC^*$. Assume
 that both $\cE_\ell(G,s)$ and $\cE_\ell(C,1)$ form single $\ell$-blocks $B$,
 $b$ respectively, with isomorphic defect groups. 
 Then $B$ is not a minimal counterexample to~$\Ca$. If in addition
 $l(B)\ge l(b)$ then $B$ is not a minimal counterexample to~$\Cb$ either.
\end{lem}

\begin{proof}
We have that $\Irr(B)=\coprod_t\cE(G,st)$ and $\Irr(b)=\coprod_t\cE(C,t)$
where both disjoint unions run over $C^*:=C_{G^*}(s)$-conjugacy classes of
$\ell$-elements $t$ in $C^*$. For any such $t$ the Jordan decompositions in $G$
as well as in $C$ establish bijections from both $\cE(G,st)$ and $\cE(C,t)$ to
the same Lusztig series $\cE(C_t,1)$ of $C_t:=C_{G^*}(st)=C_{C^*}(t)$. Thus,
if $\chi\in\cE(G,st)$ and $\chi'\in\cE(C,t)$ correspond to
the same character in $\cE(C_t,1)$ we have that
$\chi(1)/\chi'(1)=|G^*:C_{G^*}(s)|_{p'}$ is constant. So there is a height
preserving bijection from $\Irr(B)$ to $\Irr(b)$, whence $k(B)=k(b)$ and
$k_0(B)=k_0(b)$. As the defect groups of $B,b$ are isomorphic by assumption,
this yields the first claim. If in addition $l(B)\ge l(b)$, then
$k(B)/l(B)\le k(b)/l(b)$, thus $B$ cannot be a minimal counterexample to~$\Cb$
either.
\end{proof}

\subsection{Estimates for numbers of multipartitions}

For integers $b,w\ge0$, let $k(b,w)$ denote the number of $b$-multipartitions
of $w$. These occur in the expression of block theoretic invariants of classical
groups, and we will need upper and lower bounds for them. Firstly, by
\cite[Lemma~11.11]{Ol93} and \cite[Lemma~4.6]{MN06} they satisfy:

\begin{lem}   \label{lem:k(a,b)}
 For all $b\ge3$ and $w,w_1,w_2\ge1$  we have:
 \begin{enumerate}
  \item[\rm(a)] $k(b,w)\le b^w$, and
  \item[\rm(b)] $k(b,w_1+w_2)\le k(b,w_1)k(b,w_2)$.
 \end{enumerate}
\end{lem}

We will need the following improvement of the first assertion:

\begin{lem}   \label{lem:k(b,w)}
 Let $b\ge4$ and $w\ge5$. Then we have:
 \begin{enumerate}
  \item[\rm(a)] $k(bx,w)\le (bx)^w x^{-0.73\,w/\ln(x)}$ for all $x\ge5$; and
  \item[\rm(b)] $k(b,w)\le b^{w-0.47\,w/\ln(b)}$ unless $b=4$, $w\le10$ or
   $b=5$, $w\le7$.
 \end{enumerate}
\end{lem}

\begin{proof}
We first claim that
$$\binom{x+w-1}{w}\le x^{w-0.73w/\ln(x)}$$
for all $w\ge5$, $x\ge5$, unless $w=5$, $x\le7$, or $(w,x)=(6,5)$. Indeed, this
holds for $w=5$, $x\ge8$, by a direct check, as well as for the cases $w=6$,
$x=6,7$, and $(w,x)=(7,5)$. Now
$$\binom{x+w}{w+1}=\frac{x+w}{w+1}\binom{x+w-1}{w}$$
and
$$\frac{x+w}{w+1}=1+\frac{x-1}{w+1}\le 1+\frac{x-1}{6}\le \frac{x}{3}
  \le x^{1-1/\ln(x)}$$
for $w,x\ge5$, so we conclude by induction on $w$.
\par
Now consider (a). Clearly the number $k(bx,w)$ of $bx$-tuples of partitions of
$w$ can be split up as the sum of $k(b,i_1)\cdots k(b,i_x)$ over all
compositions $(i_1,\ldots,i_x)\mdash w$ of $w$ of length~$x$. Thus by
Lemma~\ref{lem:k(a,b)}(a) we have
$$k(bx,w)=\sum_{(i_1,\ldots,i_x)\mdash w}k(b,i_1)\cdots k(b,i_x)
  \le\sum_{(i_1,\ldots,i_x)\mdash w}b^w=\binom{x+w-1}{w}b^w$$
and the assertion follows by our previous claim except in the excluded cases
$w=5$, $x\le7$, or $(w,x)=(6,5)$. In those four cases the inequality can be
verified directly.  \par
For part~(b), we use that $k(x,w)\le\binom{x+w-1}{w}\al^w$ by
\cite[p.~43]{Ol84}, where $\al=(1+\sqrt{5})/2$. Now the first part of the
proof can be recycled to show that even
$$\binom{b+w-1}{w}\le b^{w-0.952\,w/\ln(b)}$$
unless $b=4$, $w\le15$, or $b=5$, $w\le10$, or $b=6$, $w\le9$, or $b=7$,
$w\le8$.
As $\ln(\al)<0.482$ this proves the assertion under those conditions.
The finitely remaining cases can be checked directly.
\end{proof}

\begin{lem}   \label{lem:lem3}
 Let $\ell\ge5$ be a prime, $a\ge1$ and $d\le\sqrt{\ell^a-1}$ be a divisor of
 $\ell-1$. Then for all multiples $w$ of $\ell$,
 $$k(d+(\ell^a-1)/d,w)\le\begin{cases}
   \ell^{aw-0.83\,w\log_\ell d}& \text{when $d>1$},\\
   \ell^{aw-0.9\,w/\ln(\ell)}& \text{when $d=1$, $a\ge2$,}\\
   \ell^{aw-0.57w/\ln(\ell)}& \text{when $d=a=1$},
 \end{cases}$$
 except when $d\le2$, $\ell^a=w=5$, or $(\ell^a,d,w)=(25,1,5)$.
\end{lem}

\begin{proof}
Set $b:=d+(\ell^a-1)/d$. First assume that $d>1$ and we are not in one of the
excluded cases in Lemma~\ref{lem:k(b,w)}(b). Then we have
$$k(b,w)\le b^{w-0.47w/\ln(b)}=b^w/\ell^{0.47w/\ln(\ell)}.$$
As $b=(d^2+\ell^a-1)/d\le 2\ell^a/d$ this shows that
$$k(b,w)\le(2\ell^a/d)^w\,\ell^{-0.47w/\ln(\ell)}
  =\ell^{w(a+\log_\ell2-\log_\ell d-0.47/\ln(\ell))}
  \le\ell^{aw-0.83w\log_\ell d}$$
when $d\ge4$. For $d=2,3$, actually $b\le 1.62\,\ell^a/d$ and we can conclude
as before.
\par
In the excluded case $b=4$, $6\le w\le10$, we necessarily have $\ell^a=5$,
$d=2$, and the claim is checked directly. Similarly, when $b=5$, $w\le7$ then
$\ell^a=7$, $d=2$, and again the inequality holds.
\par
For $d=1$ with $a>1$ we argue as in the proof of Lemma~\ref{lem:k(b,w)}(a),
avoiding the case that $(\ell^a,d,w)=(25,1,5)$,
while for $a=1$ we recycle the arguments in the proof of part~(b) of that
result, and first show that even
$$\binom{x+w-1}{w}\le x^{w-1.052w/\ln(x)}$$
for all $w\ge x\ge11$, as well as for $w\ge10$ when $x=7$ and for $w\ge15$ when
$x=5$. The finitely remaining cases are again checked directly.
\end{proof}
 
For integers $\ell,a,w\ge1$ and a divisor $d$ of $\ell-1$ let us define
$$k(\ell,a,d,w)
  := \sum_\bw k(d+(\ell^a-1)/d,w_0)\prod_{i\ge1}k((\ell^a-\ell^{a-1})/d,w_i)$$
where the sum runs over all \emph{$\ell$-compositions of $w$}, that is, all
tuples $\bw=(w_0,w_1,\ldots)$ of non-negative integers satisfying
$\sum w_i\ell^i =w$. We write $p_\ell(w)$ for the number of such.

\begin{lem}   \label{lem:k(B)}
 Let $\ell\ge5$, let $a,w\ge1$ with $\ell|w$, let $d<\ell^a-1$ be a divisor
 of $\ell-1$. Then
 $$k(\ell,a,d,w)\le\begin{cases}
     p_\ell(w)\,\ell^{aw-0.83\,w\ln(d)/\ln(\ell)}& \text{if }d>1,\\
     p_\ell(w)\,\ell^{aw-0.9\,w/\ln(\ell)}& \text{if }d=1,\,a\ge2,\\
     p_\ell(w)\,\ell^{w-0.57\,w/\ln(\ell)}& \text{if }d=a=1.\end{cases}$$
unless $(\ell^a,d,w)=(5,1,5)$.
\end{lem}

\begin{proof}
Observe that $\sum_{i\ge1}w_i\le(w-w_0)/\ell$, and $w_0$ is a multiple of
$\ell$ (as $\ell$ divides $w$) for all $\ell$-compositions $(w_0,w_1,\ldots)$
of $w$.
\par
We first treat the case that $d=1$. Then by Lemmas~\ref{lem:lem3}
and~\ref{lem:k(a,b)}(a) we get, with $c=0.9$ when $a>1$ and $c=0.57$ when
$a=1$, that $k(\ell,a,1,w)$ is bounded above by
$$\sum_{\bw} k(\ell^a,w_0) \prod_{i\ge1}k(\ell^a,w_i)
  \le\sum\ell^{aw_0-cw_0/\ln(\ell)+a(w-w_0)/\ell}
     +\sum k(\ell^a,w_0)\ell^{a(w-w_0)/\ell},$$
where the first sum ranges over $\ell$-compositions of $w$ with
$w_0\ge\ell$ (respectively $w_0\ge10$ when $\ell=5$, $a\le2$), and the second
one over those with $w_0=0$ ($w_0\le5$ respectively). Now note that
$a(w_0+(w-w_0)/\ell)-cw_0/\ln(\ell)\le aw-cw/\ln(\ell)$, and
$aw/\ell\le aw-cw/\ln(\ell)$ as well as
$$k(25,5)\,5^{(w-5)/5}\le 5^{2w-0.91w/\ln(5)}\quad\text{ and }$$
$$k(5,5)\,5^{(w-5)/5}\le 5^{w-0.57w/\ln(5)}$$
for all $w\ge10$. Thus we can bound all summands above by
$\ell^{w-cw/\ln(\ell)}$ to find
$$k(\ell,a,1,w)\le p_\ell(w)\, \ell^{aw-cw/\ln(\ell)}$$
unless $w=\ell=5$, $a\le2$. For $a=2$ the inequality still holds, while it
fails for $a=1$.
\par
Now consider the case when $d>1$. If $a=1$ then we may assume that
$d\le\sqrt{\ell-1}$ as the value of $b:=d+(\ell^a-1)/d$ is invariant with
respect to replacing $d$ by $(\ell-1)/d$. In particular, we thus can take
$d\le\sqrt{\ell^a-1}$ in all cases. By Lemma~\ref{lem:lem3} we obtain
$$\begin{aligned}
  k(\ell,a,d,w)\le &\sum_{\bw} k(d+(\ell^a-1)/d,w_0)
     \prod_{i\ge1}k((\ell^a-\ell^{a-1})/d,w_i)\\
  \le&\sum \ell^{aw_0-0.79w_0\log_\ell d}(\ell^a/d)^{(w-w_0)/\ell}
     +\sum k(b,w_0)(\ell^a/d)^{(w-w_0)/\ell},
\end{aligned}$$
where again the first sum ranges over the $\ell$-compositions with
$w_0\ge\ell$ (respectively $w_0\ge10$ when $\ell^a=5$), and the second one over
those with $w_0\le5$. Now note that
$$aw_0-0.83w_0\log_\ell d+(w-w_0)(a-\log_\ell d)/\ell\le aw-0.83w\log_\ell d
  \quad\text{for }w_0\ge\ell,$$
and also $w/(a-\log_\ell d)/\ell\le aw-0.83w\log_\ell d$, as well as
$$k(4,5)\,(5/2)^{(w-5)/5}\le 5^{w-0.83w\log_5 2}\quad\text{ for }w\ge10$$
(for the summands with $w_0=5$ in the case $\ell^a=5$, $d=2$), so we obtain
$$k(\ell,a,d,w)\le p_\ell(w)\,\ell^{aw-0.83\,w\log_\ell d},$$
as claimed. When $\ell^a=w=5$ and $d=2$, then
$$k(5,1,2,5)=k(4,5)+k(2,1)=254<2\cdot 5^{5(1-0.83\log_5 2)}$$
by explicit calculation.
\end{proof}

We will also need the following lower bound.

\begin{lem}   \label{lem:lower}
 Let $b\ge c\ge w\ge1$. Then $k(b,w)\ge \binom{c}{w}y^w$, where
 $y=\flr{b/c}$.
\end{lem}

\begin{proof}
We have $b=cy+r$ for some $0\le r\le c-1$. In the expression
$$k(b,w)=k(cy+r,w)
  =\sum_{(i_1,\ldots,i_c)\mdash w}k(y,i_1)\cdots k(y,i_{c-1})k(y+r,i_c)$$
we only consider those summands indexed by compositions $(i_1,\ldots,i_c)$ of
$w$ with exactly $w$ entries equal to~1 and all others~0. Since there are
exactly $\binom{c}{w}$ of those, this yields the stated lower bound.
\end{proof}

\subsection{The conjecture for $\GL_n(q)$ and $\GU_n(q)$}

To deal with the general linear and unitary groups, the following
elementary observation will allow us to estimate the number of conjugacy
classes in defect groups:

\begin{lem}   \label{lem:k(D)}
 Let $m\ge1$, $i\ge0$ and set $D_{i,m}:=C_m\wr C_\ell\wr\cdots\wr C_\ell$ the
 iterated wreath product with $i$ factors of the cyclic group $C_\ell$ of
 order~$\ell$. Then
 $$k(D_{i,m})\ge m^{\ell^i}/\ell^{(\ell^i-1)/(\ell-1)}$$
 and
 $$k(D_{i,m}')\ge m^{\ell^i-1}/\ell^{(\ell^i-1)/(\ell-1)+i-1}\qquad
   \text{for $i\ge1$}.$$
\end{lem}

\begin{proof}
We argue by induction on $i$. The statement is clear when $i=0$. For $i>0$ we
have $D_{i,m}= D_{i-1,m}\wr C_\ell$; by the inductive assumption the base group
has at least $m^{\ell^i}/\ell^{\ell(\ell^{i-1}-1)/(\ell-1)}$ conjugacy classes,
and the cyclic group on top has orbits of length at most~$\ell$ on these. The
first claim follows. For the second one, by \cite[Lemma~1.4]{Ol76} we have
that $D_{i,m}'$ lies in the base group $D_{i-1,m}^\ell$ of the outer wreath
product, and has index $m\ell^i$ therein. Using the first assertion this yields
the stated lower bound.
\end{proof}

\begin{prop}   \label{prop:GLn}
 Let $\ell\ge5$ be a prime not dividing~$q$. The unipotent $\ell$-blocks of
 $\GL_n(q)$ and of $\GU_n(q)$ do not provide counterexamples to~$\Ca$ or~$\Cb$.
\end{prop}

\begin{proof}
Let $d$ be the order of $q$ modulo~$\ell$. Let $B$ be a unipotent $\ell$-block
of $\GL_n(q)$. According to the reduction given in \cite[Thm.~1.9]{MO83} there
exists $w\ge0$ such that all relevant block theoretic invariants of $B$ are the
same as those of the principal $\ell$-block of $\GL_{wd}(q)$; $w$ is then
called the \emph{weight} of $B$. Hence we may and will now assume that $B$ is
the principal block of $G:=\GL_{wd}(q)$. Let $\ell^a$ be the precise
power of $\ell$ dividing $q^d-1$. Write $w=\sum_{i=0}^va_i\ell^i$ for the
$\ell$-adic decomposition of $w$. The Sylow $\ell$-subgroups of $G$ are
of the form $\prod_{i=0}^v D_i^{a_i}$ with $D_i:=D_{i,\ell^a}$ from
Lemma~\ref{lem:k(D)}, so $k(D)=\prod_i k(D_i)^{a_i}$. 
According to \cite[Prop.~6]{Ol84} the invariant $k(B)$ is given by
$$k(B)=\sum_{\bw} k(b,w_0) \prod_{i\ge1}k(b_1,w_i)=k(\ell,a,d,w),$$
where $b=d+(\ell^a-1)/d$, $b_1=(\ell^a-\ell^{a-1})/d$, and the sum runs over
the set of $\ell$-compositions $\bw$ of $w$. Clearly, $w_0\ge a_0$, and
by Lemma~\ref{lem:k(a,b)}(b) we have $k(b,w_0)\le k(b,a_0)k(b,w_0-a_0)$. It
follows that we may assume $a_0=0$ when proving $\Cb$ by using that
$k(b,a_0)\le b^{a_0}\le\ell^{aa_0}=k(D_0)^{a_0}$. Similarly, by
\cite[p.~46]{Ol84} we have $k_0(B)=\prod_{i\ge0}k(b\ell^i,a_i)$ and thus we may
also assume $a_0=0$ when proving $\Ca$ by cancelling $k(b,a_0)$ from the upper
bound for $k(B)$ and from the expression for $k_0(B)$.
\par
So from now on, we assume $w=\sum_{i=1}^va_i\ell^i$ is divisible by $\ell$.
Then by Lemma~\ref{lem:k(D)}
$$k(D)=\prod_{i=1}^v k(D_i)^{a_i}
  \ge \prod_{i=1}^v \left(\ell^{a\ell^i}/\ell^{(\ell^i-1)/(\ell-1)}\right)^{a_i}
  = \ell^{aw-\sum a_i(\ell^i-1)/(\ell-1)}
  \ge \ell^{aw-w/(\ell-1)}.$$
Furthermore, the unipotent characters in $B$ form a basic set for $B$ (see
e.g.~\cite{GH91}), and they are in bijection with $d$-multipartitions of $w$
(see \cite{BMM,CE94}), whence $l(B)=k(d,w)\ge k(1,w)$, the number of partitions
of $w$. As for $k(B)$, by Lemma~\ref{lem:k(B)}
$$k(B)\le \begin{cases}
  p_\ell(w)\, \ell^{aw-0.83w\ln(d)/\ln(\ell)}& \text{ for }d>1,\\
  p_\ell(w)\, \ell^{aw-0.9w/\ln(\ell)}& \text{ if }a>1,\,d=1,\\
  p_\ell(w)\, \ell^{w-0.57w/\ln(\ell)}& \text{ for }a=d=1,
 \end{cases}$$
unless $(\ell^a,d,w)=(5,1,5)$. Now observe that $p_\ell(w)\le k(1,w)\le l(B)$.
Combining our estimates we then indeed obtain
$$k(B)\le |W|\, \ell^{w(a-1/(\ell-1))}\le l(B)\ k(D),$$
as required for $\Cb$. (This also holds in the case $(\ell^a,d,w)=(5,1,5)$ as
there $k(B)=k(5,5)+k(4,1)=510$, $l(B)=k(1,5)=7$ and $k(D)\ge 625$.)
\par
Now consider $\Ca$. As Sylow $\ell$-subgroups of $G$ are of the form
$\prod_{i=0}^v D_i^{a_i}$, we have $k(D')=\prod k(D_i')^{a_i}$.
By Lemma~\ref{lem:k(D)}
$$k(D')=\prod_{i=1}^v k(D_i')^{a_i}\ge
  \prod_{i=1}^v\left(\ell^{a(\ell^i-1)-(\ell^i-1)/(\ell-1)-(i-1)}\right)^{a_i}
  = \ell^{aw-w/(\ell-1)-\sum a_i(a+i-\ell/(\ell-1))}.$$
Furthermore, using Lemma~\ref{lem:lower}, we have
$$k_0(B)=\prod_{i\ge1}k(b\ell^i,a_i)\ge\prod_{i\ge1}(b\ell^{i-1})^{a_i}
  =(b/\ell^a)^{\sum a_i}\ell^{\sum a_i(a+i-1)}.$$
Let first $d=1$. Thus $b=\ell^a$. We also use from \cite[Lemma~5.2]{MkB} that
$p_\ell(w)\le\ell^{\binom{u+1}{2}}$ with $u=\flr{\log_\ell(w)}$. Then with the
estimates from Lemma~\ref{lem:k(B)} we find
$$\begin{aligned}
  k(D')k_0(B)/k(B)\ge&\ell^{aw-w/(\ell-1)-\sum a_i(a+i-\ell/(\ell-1))
   +\sum a_i(a+i-1)-aw+0.9w/\ln(\ell)}/p_\ell(w)\\
  \ge& \ell^{0.9w/\ln(\ell)-u(u+1)/2-w/(\ell-1)}>1
\end{aligned}$$
for $a>1$, and
$$k(D')k_0(B)/k(B)
  \ge\ell^{0.57w/\ln(\ell)-u^2-w/(\ell-1)+\sum a_i/(\ell-1)}>1$$
for $(\ell,w)\ne(5,5)$ when $a=1$. The case $\ell=w=5$ can be checked directly.
\par
Now consider the case when $d>1$, where with $b\ge \ell^a/d$ we have
$$k_0(B)\ge (b/\ell^a)^{\sum a_i}\ell^{\sum a_i(a+i-1)}
  \ge\ell^{\sum a_i(a+i-1)}d^{-\sum a_i}
  =\ell^{\sum a_i(a+i-1)-\sum a_i\log_\ell d}.
$$
Combining the above estimates we find
$$k(D')k_0(B)/k(B)\ge\ell^{(0.83w-\sum a_i )\log_\ell d-w/(\ell-1)
  +\sum a_i/(\ell-1)-u(u+1)/2}\ge2$$
unless $d=2$, and either $\ell=5$, $w\in\{5,10,25,30,35\}$, or $\ell=w=7$.
(The stronger bound~2 will be needed in the proof of Theorem~\ref{thm:class}.)
For $w=\ell=7$ we get $k(B)=2996$, $k_0(B)=35$ and $k(D')\ge 7^5$;
for $\ell=5$ and $w\in\{10,25,30\}$, replacing our bound for $p_5(w)$ by the
actual values~3, 7 and 9, respectively shows the claim, and finally for $w=5$
we have $k(5,1,2,5)=254$, $k_0(B)=20$ and $k(D')\ge5^3$.
\par
Now let $B$ be a unipotent $\ell$-block of $\GU_n(q)$, and write $d$
for the order of $-q$ modulo~$\ell$. Then the block theoretic invariants
of $B$ are the same as for the principal $\ell$-block of $\GL_{wd}(q')$,
where $w$ is the weight of $B$ and $q'$ is any prime (power) such that
$q'$ has order $d$ modulo~$\ell$ (see \cite{Ol84}). Thus the claim for $B$
follows from the result for blocks of $\GL_n(q)$ proved above.
\end{proof}

\subsection{Block invariants of special linear and unitary groups}

We now turn to the quasi-simple groups $\SL_n(q)$ and $\SU_n(q)$.
For this we need to recall in some detail Olsson's description
\cite[p.~45]{Ol84} of the set of characters in the principal $\ell$-block of
$\GL_n(\eps q)$, $\eps=\pm1$. (As customary we write $\GL_n(-q):=\GU_n(q)$, and
similarly $\SL_n(-q):=\SU_n(q)$ and so on.) If $\ell$ divides $q-\eps$, the
principal $\ell$-block is the unique unipotent block, so by Brou\'e--Michel
it consists of the union of the Lusztig series indexed by conjugacy classes of
$\ell$-elements in $\GL_n(\eps q)$. A conjugacy class of $\GL_n(q)$ is uniquely
determined by the characteristic
polynomial of its elements, which is a product of minimal polynomials over
$\FF_q$ of elements of $\ell$-power order in $\overline{\FF}_q^\times$. Let
$\cF^i$ denote the set of such polynomials of degree $\ell^i$, where $i\ge0$.
Then $|\cF^0|=\ell^a$ and $|\cF^i|=\ell^a-\ell^{a-1}$ for $i>0$, where
we let $\ell^a$ denote the precise power of $\ell$ dividing $q-1$. The
conjugacy classes of $\ell$-elements in $\GL_n(q)$ are thus in bijection
with maps
$$m: \cF:=\bigcup\cF^i \rightarrow\ZZ_{\ge0},\quad f\mapsto m_f,\qquad
  \text{ with }\quad\sum_i\sum_{f\in\cF^i} m_f\ell^i=n,$$
in such a way that $m$ labels the class of $\ell$-elements with characteristic
polynomial $\prod_f f^{m_f}$. If $t\in\GL_n(q)$ corresponds to $m$, then
the Lusztig series $\cE(G,t)$ contains $\prod_f k(1,m_f)$ characters
(all lying in the principal $\ell$-block). It is shown from this in
\cite[Prop.~6]{Ol84} that $k(B)=k(\ell,a,1,n)$.

We now determine the number of characters of height zero in the principal
$\ell$-block of $\SL_n(\eps q)$ for $\ell|(q-\eps)$; the formula for the
number of modular characters is an immediate consequence of the main result of
Kleshchev and Tiep \cite{KT09}:

\begin{thm}   \label{thm:k0 SL}
 Let $\SL_n(\eps q)\le G\le\GL_n(\eps q)$ with $\eps\in\{\pm1\}$, and $\ell>2$
 be a prime dividing $q-\eps$. Set $\ell^a:=(q-\eps)_\ell$,
 $\ell^g:=|\GL_n(\eps q):G|_\ell$ and $\ell^u=\gcd(\ell^g,n,q-\eps)$.
 Let $B$ denote the principal $\ell$-block of $G$, and $\tB$ the principal
 $\ell$-block of $\GL_n(\eps q)$. Then
 $$k_0(B)=k_0(\tB)/\ell^g
    +\begin{cases} \ell^{a+f-g}& \text{ if $n=\ell^f$ and $g>0$,}\\
                   0& \text{ else},
 \end{cases}$$
 and
 $$l(B)= k(1,n)+\sum_{i=1}^u (\ell^i-\ell^{i-1})\,k(1,n/\ell^i).$$
\end{thm}

\begin{proof}
Let first $\eps=1$, so $\SL_n(q)\le G\le\GL_n(q)=:\tG$. The characters in the
principal $\ell$-block $B$ of $G$ are precisely the constituents of the
restrictions to $G$ of the characters in the principal $\ell$-block $\tB$ of
$\tG$. If $\chi\in\Irr_0(\tB)$ then, being of $\ell'$-degree, by Clifford theory
it cannot split upon restriction to a
normal subgroup of $\ell$-power index, so the restrictions of characters in
$\Irr_0(\tB)$ contribute $k_0(\tB)/\ell^g$ to $k_0(B)$. Any further height zero
character of $B$ must be the constituent of a splitting character in $\tB$, of
degree divisible by $\ell$.  Assume $\chi\in\Irr(\tB)$ lies in the Lusztig
series $\cE(\tG,t)$ of the $\ell$-element $t\in\GL_n(q)$. As argued in the
proof of \cite[Thm.~5.1]{MkB} its restriction to $\SL_n(q)$ splits into
$A_t:=|C_{\PGL_n}(\bar t)^F:C_{\PGL_n}^\circ(\bar t)^F|$ distinct constituents,
where $\bar t$ is the image of $t$ in $\PGL_n$ and $F$ is the standard
Frobenius endomorphism on $\PGL_n$. By Clifford theory $\chi$ then splits into
$\ell^b=\min\{A_t,\ell^g\}$ distinct constituents upon restriction to~$G$.
Thus $\chi|_G$ contributes to $\Irr_0(B)$ if and only if $\chi(1)_\ell=\ell^b$.
Now $\ell^i|A_t$ if the set of
eigenvalues of $t$ is invariant under multiplication by an $\ell^i$th root of
unity $\zeta\in\FF_q^\times$. So $\ell$-elements $t$ with $\ell^i|A_t$ are
parametrised by maps $m:\cF\rightarrow\ZZ_{\ge0}$ as above that are constant on
$\zeta$-orbits. For $m$ such a map,
$C_{\tG}(t)\cong\prod_{f\in\cF}\GL_{m_f}(q^{\deg(f)})$. Thus, if $m_f>0$ for
some non-linear polynomial $f\in\cF$ (and hence with $\deg(f)\ge\ell$), then
$|\tG:C_{\tG}(t)|_\ell\ge\ell^{a(\ell-1)}>\ell^a$, whence the Lusztig series of
$t$ cannot contribute to $\Irr_0(B)$. To count characters of height zero
we thus only need to consider $\ell$-elements $t$ that are diagonalisable over
$\FF_q$.
\par
Let $m$ be the map corresponding to such an element, with support on linear
polynomials in $\cF$ and consider the $\ell$-adic decompositions
$m_f=\sum_{i\ge0}m_{f,i}\ell^i$ for $f\in\cF$ and $n=\sum_i n_i\ell^i$. Then as
discussed in \cite[proof of Cor.~2.6]{MO83} every index $i$ for which
$\sum_f m_{f,i}-n_i=k>0$ contributes a factor of at least $(k!)_\ell$ to the
index of $C_{\tG}(t)$ in $\GL_n(q)$ and thus to the degree of any
$\chi\in\cE(\tG,t)$. If $\chi$ splits into $\ell^i$ factors, then as seen
before all
multiplicities $m_f$ occur a multiple of $\ell^i$ times, and by the
preceding discussion each such will contribute a factor $(\ell^i!)_\ell$ to
$\chi(1)_\ell$. This is bigger than $\ell^i$ if $i>1$ or if there are at
least two such orbits of polynomials. Thus contributions to $\Irr_0(B)$ can
occur in this way only when first $n=\ell^f$ is a power of $\ell$, and second,
$g\ge1$ and $i=1$. In this case, there are exactly $\ell^{a-1}$ orbits of
linear polynomials over $\FF_q$, and hence the same number of maps $m$. The
centraliser of a corresponding $\ell$-element has the structure
$C_{\tG}(t)\cong\GL_{n/\ell}(q)^\ell$. Now $\GL_{n/\ell}(q)$ has precisely
$n/\ell=n^{f-1}$ unipotent characters of height zero, each leading to $\ell$
height zero characters. Any of these necessarily restricts irreducibly to
$G$. Thus we obtain $\ell^{a-1}\ell^{f-1}\ell/\ell^{g-1}=\ell^{a+f-g}$ further
characters in $\Irr_0(B)$, as claimed.
\par
Let us now consider the number $l(B)$. Kleshchev and Tiep \cite[Thm.~1.1]{KT09}
describe the number of irreducible constituents of the restriction to
$\SL_n(q)$ of an irreducible $\ell$-modular Brauer character of $\GL_n(q)$, as
follows: The Brauer characters in the principal block
are indexed by partitions $\la\vdash n$, and the character labelled by $\la$
splits into $\gcd(q-1,\la_1',\la_2',\ldots)_\ell$ constituents for $\SL_n(q)$,
where $\la'=(\la_1',\la_2',\ldots)$ is the partition conjugate to $\la$.
Now clearly for any divisor $\ell^i$ of $n$, the partitions of $n$ with all
parts divisible by $\ell^i$, are in natural bijection with partitions of
$n/\ell^i$, whence their number is $k(1,n/\ell^i)$. By Clifford theory, the
splitting occurs `at the top' of the chain of normal subgroups of $\ell$-power
index, whence the claimed formula for $l(B)$ by induction.
\par
Now let $\eps=-1$. As pointed out before, the block theoretic invariants for
the principal $\ell$-block of $\GU_n(q)$ coincide with those of the principal
$\ell$-block of $\GL_n(q^2)$, and a Sylow $\ell$-subgroup of $\GU_n(q)$
becomes a Sylow $\ell$-subgroup of $\GL_n(q^2)$ under the natural embedding
$\GU_n(q)\le\GL_n(q^2)$. The proof for $\GU_n(q)$ is then entirely similar to
the above, where for the second part we replace the reference to \cite{KT09}
by the corresponding result \cite[Prop.~4.9]{De17} of Denoncin for the groups
$\SU_n(q)$.
\end{proof}

The situation for the projective general linear and unitary groups is somewhat
easier:

\begin{prop}   \label{prop:k0 PGL}
 Let $\ell>2$ be a prime dividing $q-\eps$. Set $\ell^a=(q-\eps)_\ell$ and
 $\ell^m=\gcd(n,q-\eps)_\ell$. Then the principal $\ell$-block $\bar B$ of
 $\PGL_n(\eps q)$ satisfies $k_0(\bar B)=k_0(\tilde B)/\ell^{a-m}$,
 where $\tB$ is the principal $\ell$-block of $\GL_n(\eps q)$.
\end{prop}

\begin{proof}
We present the argument in the case $\eps=1$, the case $\eps=-1$ again being
entirely analogous due to the agreement of block theoretic invariants
between $\GU_n(q)$ and $\GL_n(q^2)$. As $\PGL_n(q)=\GL_n(q)/Z(\GL_n(q))$ the
characters in $\Irr_0(\bar B)$ are the
characters in $\Irr_0(\tB)$ having $Z(\tG)$ in their kernel, with $\tB$ the
principal $\ell$-block of $\tG=\GL_n(q)$. Let $\chi\in\Irr_0(\tB)$. Then it
lies above a character of $\SL_n(q)$ having $Z:=Z(\SL_n(q))_\ell$ in its kernel
(as all characters above non-trivial central characters of $\ell$-power order
have degree divisible by $\ell$). Furthermore, $\del\otimes\chi\in\Irr_0(\tB)$
for all linear characters $\del$ of $\tG$ of $\ell$-power order. Assume that
$\del\otimes\chi=\chi$ for some $\del\ne1$. Then $\chi$ is induced from some
subgroup of $\tG$ of index divisible by $\ell$, which is not possible as
$\chi(1)$ is prime to $\ell$. Tensoring by the $\ell^a$ linear characters
of $\ell$-power order permutes the various central characters modulo $Z$,
hence exactly $|Z|_\ell=\ell^m$ out of each orbit on $\Irr_0(\tB)$ lie above
the trivial character of $Z(\tG)$, whence $k_0(\bar B)=k_0(\tB)/\ell^{a-m}$. 
\end{proof}

\begin{exmp}   \label{exmp:k(B)}
(a) Assume that $n$ is not divisible by $\ell$. Then in the situation and
notation of Theorem~\ref{thm:k0 SL} we have $g=u=0$, and hence
$$k_0(B)=k_0(\tilde B)/\ell^g\qquad\text{and}\qquad l(B)=l(\tilde B),$$
as expected (as restrictions of characters in $\ell$-element Lusztig series are
irreducible in this case), and $k_0(\bar B)= k_0(\tB)/\ell^a$ in the situation
of Proposition~\ref{prop:k0 PGL}.  \par
(b) Next assume that $n=\ell$; in this case the only proper subgroup $G$ of
$\GL_n(\eps q)$ allowed in Theorem~\ref{thm:k0 SL} is $G=\SL_\ell(\eps q)$.
Here, $g=u=1$, and thus
$$k(B)=\big(k(\ell,a,1,\ell)+(\ell^2-1)\,k(\ell,a,1,1)\big)/\ell^a
  =\big(k(\ell^a,\ell)+\ell^{a+2}-\ell^{a-1}\big)/\ell^a$$
  by \cite[Thm.~5.1]{MkB},
$$k_0(\tB)=k(\ell^{a+1},1)=\ell^{a+1},\ k_0(B)=k_0(\tB)/\ell^a+\ell=2\ell,
  \ k_0(\bar B)=k_0(\tB)/\ell^{a-1}=\ell^2,$$
by Theorem~\ref{thm:k0 SL} and Proposition~\ref{prop:k0 PGL}, and
$$l(B)=k(1,\ell)+\ell-1.$$
\end{exmp}

\subsection{The conjecture for $\SL_n(q)$ and $\SU_n(q)$}
To show our inequalities for blocks of $\SL_n(\eps q)$ we will need to
control the derived subgroup of a Sylow subgroup:

\begin{prop}   \label{prop:Syl SLn}
 For $\ell>2$ with $\ell|(q-1)$, let $\tilde D$ be a Sylow $\ell$-subgroup of
 $\GL_n(q)$ and $D=\tilde D\cap\SL_n(q)$, a Sylow $\ell$-subgroup of $\SL_n(q)$.
 Then
 $$|\tD':D'|=\begin{cases} \ell& \text{ if $n=\ell^f$ for some $f\ge1$},\\
   1& \text{else.}\end{cases}$$
\end{prop}

\begin{proof}
Let $n=\sum_{i\ge0}a_i\ell^i$ be the $\ell$-adic decomposition of $n$, and
$\ell^a$ the precise power of $\ell$ dividing $q-1$. Then
$\tD=\prod_{i\ge0} D_i^{a_i}$, with $D_i=D_{i,\ell^a}$, and $D$ is the
subdirect product of these factors defined by the determinant condition.
Now if that product has more than one factor, it is straightforward to see
that the derived subgroup of the subdirect product $D$ agrees with that of
$\tD$. Hence now assume that $n=\ell^i$ is a power of $\ell$. So
$\tD= D_{i-1}\wr C_\ell$. Clearly, $\tD'$ and thus also $D'$ lies in the base
group. We claim that $D'$ consists of all $\ell$-tuples $(x_1,\ldots,x_l)$ of
elements of $D_{i-1}$ such that $x_1\cdots x_\ell\in D_{i-1}'$ and
$x_1x_2^2\cdots x_{\ell-1}^{\ell-1}$ is an $\ell$th power.
Indeed, $D'$ contains any tuple with entries in $D_{i-1}'$, as well as
every tuple of the form $(x^\ell,x^{-\ell},1,\ldots,1)$ (being the commutator
of $(x^{\ell-1},x^{-1},\ldots,x^{-1})$ with an $\ell$-cycle), as well as all
tuples of the form $(x,x^{-1},x^{-1},x,1,\ldots,1)$ (being the commutator of
$(x,x^{-1},1,\ldots,1)$ with an $\ell$-cycle). According to
\cite[Lemma~1.4]{Ol76}, $\tD'$ consists of all $\ell$-tuples with
$x_1\cdots x_\ell\in D_{i-1}'$, and the group described above has
index~$\ell$ therein.
\end{proof}

\begin{thm}   \label{thm:SLn}
 Let $H=G/Z$ with $G\in\{\SL_n(q),\SU_n(q)\}$ and $Z\le Z(G)$. Assume that
 $\ell\ge 5$. Then the unipotent $\ell$-blocks of $H$ are not
 counterexamples to~$\Ca$ or~$\Cb$.
\end{thm}

\begin{proof}
We consider $G=\SL_n(q)$ as a normal subgroup of $\tilde G:=\GL_n(q)$. As
$\tG/G$ is cyclic, restriction of characters from $\tilde G$ to $G$ is
multiplicity-free. Moreover, as pointed out in the proof of
Theorem~\ref{thm:k0 SL} all characters in $\ell$-element Lusztig series
restrict irreducibly unless $\ell$ divides $\gcd(n,q-1)$.   \par
Let $B$ be a unipotent $\ell$-block of $G$ with defect group $D$. Then there
is a unipotent $\ell$-block $\tB$ of $\tilde G$ covering $B$, with defect
group $\tD\ge D$, and by Proposition~\ref{prop:GLn} both inequalities are
satisfied for $\tB$.
\par
First assume that $\ell$ does not divide $q-1$, thus $\tD=D\le G$. As
$\Irr(\tB)\subseteq\cE_\ell(\tG,1)$ the preceding discussion shows that all
$\chi\in\Irr(\tB)$ restrict irreducibly to $G$. The various characters of
$\tG$ with the same restriction to $G$ lie in the Lusztig series of the $q-1$
central $\ell'$-elements of $\tG^*=\GL_n(q)$ and thus in pairwise distinct
$\ell$-blocks of $\tG$. So we obtain $k(B)=k(\tB)$, $k_0(B)=k_0(\tB)$, and also
$l(B)=l(\tB)$ since by \cite{GH91} the unipotent characters form a basic set
for the unipotent blocks of $G$. In particular the conjecture holds for the
block $B$ of $\SL_n(q)$. Furthermore, $|Z(G)|$ is a divisor of $q-1$ and thus
prime to $\ell$ in this case, so all characters in $\Irr(B)$ have $Z(G)$ in
their kernel, and the claim also follows for $H=G/Z$ for any $Z\le Z(G)$.
\par
So now assume that $\ell|(q-1)$. Then all unipotent characters of $G$ lie in
the principal $\ell$-block (see e.g. \cite[Thm.]{CE94}), so $B$ has weight
$w=n$ and any defect group $D$ of $B$ is a Sylow $\ell$-subgroup of $G$. Let
us set $\ell^m:=\gcd(n,q-1)_\ell=\min\{w_\ell,\ell^a\}=|Z(G)|_\ell$.
Now note that $k(D)\ge k(\tD)/\ell^a\ge \ell^{a(w-1)-w/(\ell-1)}$ and
$k(\bar D)\ge k(D)/\ell^m\ge\ell^{a(w-1)-w/(\ell-1)-m}$, where $\bar D$ is
the image of $D=\tilde D\cap\SL_n(q)$ in $\PSL_n(q)$. Furthermore we have that
$l(B)\ge l(\tB)=k(1,w)$ since the restrictions of unipotent characters of
$\SL_n(q)$ to $\ell$-regular classes are linearly independent (or by
Theorem~\ref{thm:k0 SL}). To estimate $k(B)$, first observe that for any
$x\ge1$
$$k(\ell,a,1,x)=\sum_{\bw}k(\ell^a,w_0)\prod_{i\ge1}k(\ell^a-\ell^{a-1},w_i)
  \le \sum_{\bw}\ell^{aw_0}\prod_{i\ge1}\ell^{aw_i}
  \le p_\ell(x)\ell^{ax}$$
by Lemma~\ref{lem:k(a,b)}(a), where $\bw$ runs over $\ell$-compositions of~$x$.
Thus from \cite[Thm.~5.21]{MkB} we find
$$k(B)\le
  \Big(k(\ell,a,1,w)+\sum_{j=1}^m p_\ell(w/\ell^j)\ell^{2j+aw/\ell^j}\Big)
  /\ell^a.
$$
By \cite[Lemma~5.2]{MkB}, $p_\ell(w/\ell^j)\le\ell^{\binom{u+1-j}{2}}$ with
$u=\flr{\log_\ell w}$, so that with the estimates in Lemma~\ref{lem:k(B)} we
obtain
$$k(B)\le
  \Big(\ell^{aw-cw/\ln(\ell)+\binom{u+1}{2}}
    +\sum_{j=1}^m\ell^{2j+aw/\ell^j+\binom{u+1-j}{2}}\Big)/\ell^a
$$
where $c=0.9$ for $a>1$ and $c=0.57$ for $a=1$.
Now
$$2j+aw/\ell^j+\binom{u+1-j}{2}\le aw-cw/\ln(\ell)+\binom{u+1}{2}-2-j$$
for $j\ge1$ and all relevant $\ell,w$, unless $\ell^a=w\in\{5,7\}$, so
$$\begin{aligned}
  k(B)&\le \ell^{a(w-1)-cw/\ln(\ell)+\binom{u+1}{2}}
    \big(1+\sum_{j\ge1}\ell^{-2-j}\big)\\
  &\le\ell^{a(w-1)-cw/\ln(\ell)
    +\binom{u+1}{2}}\big(1+\frac{1}{\ell^2(\ell-1)}\big)
  \le\ell^{a(w-1)-cw/\ln(\ell)+\binom{u+1}{2}+d_\ell}
  \end{aligned}$$
where $d_\ell:=1/(\ell^2(\ell-1)\ln(\ell))$, with the same exceptions. As
certainly $k(\bar B)\le k(B)$ we thus find
$$\begin{aligned}
  k(\bar D)l(\bar B)/k(\bar B)
  \ge& \ell^{cw/\ln(\ell)-w/(\ell-1)-m-d_\ell}
\end{aligned}$$
which is bigger than~1 unless $\ell^a=w\in\{5,7\}$. (Observe that $m\le a$
always.) In the excluded cases, by Example~\ref{exmp:k(B)}(b) we have
$k(B)=(k(5,5)+5^3-1)/5=126$, while $l(\bar B)=k(1,5)=7$,
$k(\bar D)\ge k(D)/5= 149/5$ (respectively $k(B)=(k(7,7)+7^3-1)/7=1821$,
$l(\bar B)=k(1,7)=15$, $k(\bar D)\ge k(\tD)/49=117697/49$).
\par
As for $\Ca$, first we have by Proposition~\ref{prop:Syl SLn}
$$k(\bar D')\ge k(D')/\ell^m\ge k(\tD')/\ell^{m+1}\ge
  \ell^{aw-w/(\ell-1)-\sum a_i(a+i-\ell/(\ell-1))-m-1}$$
and
$$k_0(\bar B)=k_0(B)\ge k_0(\tilde B)/\ell^a\ge \ell^{\sum a_i(a+i-1)-a}$$
by Theorem~\ref{thm:k0 SL}, so
$$k(\bar D')k_0(\bar B)/k(\bar B)\ge
\ell^{cw/\ln(\ell)+\sum a_i/(\ell-1)-w/(\ell-1)-m-1-\binom{u+1}{2}-d_\ell}.
$$
When $a>1$ this is at least~1 unless $w=\ell\in\{5,7\}$. Now for $w=\ell$ it is
easily seen that $\bar D'$ is abelian of order $\ell^{a\ell-a-2}$, a factor of
$\ell$ bigger than our general estimate, and the desired inequality follows.
\par
Finally, when $a=1$ then we conclude with the general estimate unless
$\ell=w\le19$, or $\ell=7$, $w\le14$, or $\ell=5$, $w\le35$. The case
$\ell=w>7$ is handled as when $a>1$, and this clearly also applies when
$w=k\ell$ with $k\le\ell-1$, so we are left with $w=\ell=7$ and $\ell=5$,
$w\in\{5,25\}$. These three cases are easily checked directly.
\par
Again, the proof for $\SU_n(q)$ is entirely similar, with $d$ now denoting the
order of $-q$ modulo~$\ell$.
\end{proof}

\subsection{Groups of classical type}

We start off by determining the number of height zero characters in
unipotent $\ell$-blocks of classical type groups. Let $q$ be a prime power and
$\ell\ne2$ an odd prime not dividing $q$. We write $d=d_\ell(q)$ for the order
of $q$ modulo~$\ell$ and set $d':=d/\gcd(d,2)$. First let $G_n(q)$ be one of
$\Sp_{2n}(q)$ or $\SO_{2n+1}(q)$ with $n\ge2$, and $\bG_n$ the corresponding
simple algebraic group over $\overline{\FF}_q$. The unipotent $\ell$-blocks of
$G_n(q)$ are parametrised by $d$-cuspidal pairs in $G_n(q)$ (see
\cite[Thm.]{CE94}), that is, by pairs $(\bL,\la)$ (up to conjugation) where
$\bL$ is a $d$-split Levi subgroup of $\bG_n$ with
$L=\bL^F\cong G_{n-wd'}(q)\times T_d^w$ for some $w\ge0$, and $\la$ is a
$d$-cuspidal unipotent character of $L$, and hence of $G_{n-wd'}(q)$.
Here, $T_d$ denotes a torus $T_d\cong \GL_1(q^d)$ if $d$ is odd, respectively
$T_d\cong \GU_1(q^{d'})$ if $d=2d'$ is even. 

Now let $G_n^\eps(q)=\SO_{2n}^\eps(q)$, with
$\eps\in\{\pm\}$, $n\ge4$, be an even-dimensional orthogonal group. (Here, as
customary, we write $\SO_{2n}$ for the connected component of the identity in
the general orthogonal group $\GO_{2n}$.) The unipotent $\ell$-blocks of
$G_n^\eps(q)$ are again parametrised by $d$-cuspidal pairs $(\bL,\la)$, where
$L=G_{n-wd'}^\delta(q)\times T_d^w$, with $T_d$ as above,
and $\delta=\eps$ if $d$ is odd or $w$ is even, and $\delta=-\eps$ else, and
$\la$ is a $d$-cuspidal unipotent character of $L$. In either case we write
$b(\bL,\la)$ for the corresponding block and call $w$ its \emph{weight}. The
number $k(b(\bL,\la))$ of characters was determined in \cite[Prop.~5.4]{MkB}.
\par
By \cite[Thm.~4.4(ii)]{CE94} the defect groups of $b(\bL,\la)$ are isomorphic
to Sylow $\ell$-subgroups of $C_G([\bL,\bL])$. For $G=\Sp_{2n}(q)$ we have
$[\bL,\bL]=\Sp_{2(n-wd')}$ has centraliser $\Sp_{2wd'}(q)$ in $G$; for
$G=\SO_{2n+1}(q)$ we have $[\bL,\bL]=\SO_{2(n-wd')+1}$ with centraliser
$\GO_{2wd'}^\pm(q)$ in $G$ where the ``+'' sign occurs if and only if $d$ is
odd; and for $G=\SO_{2n}^\eps(q)$, $C_G([\bL,\bL])=\GO_{2wd'}^{\eps\delta}(q)$.
Observe that by the parity condition on the sign $\eps\delta$, a Sylow
$\ell$-subgroup of $\GO_{2wd'}^{\eps\delta}(q)$ is also a Sylow
$\ell$-subgroup of $\SO_{2wd'+1}(q)$. In any case, a Sylow $\ell$-subgroup $P$
of $C_G([\bL,\bL])$ is isomorphic to the wreath product $C_{\ell^a}\wr S$, with
$\ell^a$ the precise power of $\ell$ dividing $q^d-1$ and $S$ a Sylow
$\ell$-subgroup of the complex reflection group $G(2d',1,w)$. Now by its
definition $d'$ is not divisible by $\ell$, and thus a Sylow $\ell$-subgroup
of the wreath product $G(2d',1,w)\cong C_{2d'}\wr\fS_w$ is isomorphic to a
Sylow $\ell$-subgroup of $\fS_w$. So $P$ is isomorphic to a Sylow
$\ell$-subgroup of $\GL_{dw}(q)$ (when $d$ is odd) or $\GU_{d'w}(q)$ (if $d$ is
even).
\par
Let $\tilde B$ be a block of
$\GO_{2n}^\eps(q)$ lying above the unipotent $\ell$-block $B=b(\bL,\la)$ of
$\SO_{2n}^\eps(q)$. Then either $\tilde B$ lies above a unique unipotent block
of $\SO_{2n}^\eps(q)$, in which case the tensor product of $\tilde B$ with the
non-trivial linear character of $\GO_{2n}^\eps(q)$ is another block above $B$,
or else the cuspidal pair $(\bL,\la)$ is such that $\la$ is labelled by a
degenerate symbol, in which case $\tilde B$ lies above the two blocks
parametrised by the two unipotent characters labelled by this degenerate symbol.
In either case, the unipotent characters in $\tilde B$ are in bijection with
the irreducible characters of $G(2d',1,w)$.

\begin{thm}   \label{thm:k0 BCD}
 Let $G$ be one of $\Sp_{2n}(q)$ ($n\ge2$), $\SO_{2n+1}(q)$ ($n\ge3$), or
 $\GO_{2n}^\pm(q)$ ($n\ge4$), let $\ell\ne2$ be a prime not dividing $q$ and
 let $B$ be a unipotent $\ell$-block of $G$ of weight~$w$. Let $d=d_\ell(q)$,
 $d':=d/\gcd(d,2)$ and write $\ell^a$ for the precise power of $\ell$ dividing
 $q^d-1$. Then
 $$k_0(B)=\prod_{i\ge0} k\big((2d'+(\ell^a-1)/2d')\ell^i,a_i\big),$$
 where $w=\sum_{i\ge0}a_i\ell^i$ is the $\ell$-adic decomposition of $w$, and
 $$l(B)=k(2d',w).$$
\end{thm}

\begin{proof}
The characters of $G$ in a unipotent $\ell$-block $B=b(\bL,\la)$ are
parametrised in \cite[Prop.~5.4 and~5.5]{MkB}: Let $\chi\in\Irr(B)$ and
$t\in G^*$ an
$\ell$-element such that $\chi\in\cE(G,t)$. Then, up to conjugation we must
have that $\bL$ is a $d$-split Levi subgroup of the dual of $\bC=C_{\bG^*}(t)$,
and the Jordan correspondent $\psi\in\cE(C,1)$ of $\chi$ also lies in the
unipotent block $B_C$ of $C$ parametrised by $(\bL,\la)$. For $\chi$ to be of
height zero, it then follows from the degree formula for Jordan decomposition
that a Sylow $\ell$-subgroup of $C^*$ has to be a Sylow $\ell$-subgroup of
$C_G([\bL,\bL])$, and furthermore, $\psi$ has to be of height zero in $B_C$.
This imposes exactly the same conditions on $t,\psi$ as in the proof of the
corresponding result \cite[Prop.~2.13]{MO83} for blocks of $\GL_{2d'w}(q)$, and
thus we can conclude as there.
\par
The statement about $l(B)$ follows as the unipotent characters form a basic
set for the unipotent blocks \cite{GH91}, and the unipotent characters in $B$
are in bijection with the irreducible characters of $G(2d',1,w)$ (see
\cite{BMM,CE94}), of which there are precisely $k(2d',w)$.
\end{proof}

\begin{thm}   \label{thm:class}
 Let $H$ be quasi-simple of classical Lie type in characteristic~$p$ and
 assume that $\ell\ne2,p$. Then the unipotent $\ell$-blocks of $H$ are not
 counterexamples to~$\Ca$ or~$\Cb$.
\end{thm}

\begin{proof}
By Proposition~\ref{prop:exc cover} we need not concern ourselves with
exceptional covering groups. The special linear and unitary groups were handled
in Theorem~\ref{thm:SLn}. For the other groups of classical type the order of
the centre of any non-exceptional cover $H$ is a 2-power, and all unipotent
$\ell$-blocks have $Z(H)$ in their kernel, so we can restrict to the
case when $H$ is simple. \par
Let us first consider the orthogonal groups $H=\OO_{2n+1}(q)$ with $n\ge2$. Let
$d=d_\ell(q)$, and let $B=b(\bL,\la)$ be the unipotent $\ell$-block of
$G=\SO_{2n+1}(q)$ parametrised by the $d$-cuspidal pair $(\bL,\la)$, with
$\bL$ of semisimple rank $n-wd'$, where $d'=d/\gcd(d,2)$. As recalled above,
the Sylow $\ell$-subgroups of $C_G([\bL,\bL])$ are defect groups of $B$.
In our case, $C_G([\bL,\bL])=\GO_{2wd'}^\pm(q)$, with the ``+''-sign
occurring when $d$ is odd. Now observe that Sylow $\ell$-subgroups of
$\GO_{2wd'}^\pm(q)$ are contained in a subgroup $\GL_{wd'}(q)$ when $d$ is odd,
respectively $\GU_{wd'}(q)$ when $d$ is even. They are hence isomorphic to a
direct product $\prod_{i\ge0}D_i^{a_i}$, where $\sum_{i\ge0}a_i\ell^i$ is the
$\ell$-adic decomposition of $w$ and $D_i=D_{i,\ell^a}$ with
$\ell^a=(q^d-1)_\ell$. Assume that $d$ is odd, so $d'=d$. Choose $q'$ a prime
such that $q'$ has order $2d$ modulo~$\ell$ (which is possible as
$d$ and $\ell$ both are odd). Then the formulas in \cite[Prop.~5.4]{MkB} and
in Theorem~\ref{thm:k0 BCD} show that the principal $\ell$-block $B'$ of
$\GL_{2dw}(q')$ has the same invariants $k(B')$, $k_0(B')$, $l(B')$ and $D$
as the block~$B$. Similarly, if $d=2d'$ is even, then the principal
$\ell$-block of $\GL_{dw}(q)$ has the same invariants as $B$. Thus both $\Ca$
and $\Cb$ for $B$ follow from the corresponding result for the general linear
group in Proposition~\ref{prop:GLn}. The derived subgroup $H=[G,G]$
has index~2 in $G$, and by Lusztig's result \cite[??]{CE} characters in
$\ell$-series
restrict irreducibly to $H$, so $k(B')=k(B)$ for the unipotent $\ell$-block
$B'$ of $H$ covered by $B$, and the other invariants do not change either.
This gives the claim for $H=\OO_{2n+1}(q)$.
\par
The situation for $H=\PSp_{2n}(q)$ is entirely analogous. Finally assume that
$G=\SO_{2n}^\pm(q)$ is an even-dimensional orthogonal group, with $n\ge4$.
Let $B$ be a unipotent $\ell$-block of $G$ of weight~$w$, and $\tB$ a block
of $\GO_{2n}^\pm(q)$ covering $B$. It is shown in \cite[Cor.~5.7]{MkB} that
$k(B)\le k(\tB)$. Now according to our formulas for block invariants, the
block $\tB$ has the same invariants as the principal block $B'$ of
$\GL_{2d'w}(q')$, with $q'$ of multiplicative order $2d'$ modulo~$\ell$. As
$|\GO_{2n}^\pm(q):\SO_{2n}^\pm(q)|=2$ we certainly have $k_0(B)\ge k_0(\tB)/2$
and $l(B)\ge l(\tB)/2$. So $\Ca$ and $\Cb$ hold for $B$ if we can show that
these inequalities hold for $B'$ with $2k(B')$ in place of $k(B')$. That is,
we need to study the proof of Proposition~\ref{prop:GLn}, for the case when
$d>1$. For $\Ca$ we had already obtained this better estimate in the proof.
On the other hand, for $\Cb$ this follows directly, as in fact
$k(d,w)\ge k(2,w)\ge 2k(1,w) \ge 2p_\ell(w)2$ for all $\ell,w$ when $d\ge2$.
Thus the claim holds for the unipotent $\ell$-blocks of $G$.
\par
The restrictions of characters in $\ell$-element series to the derived subgroup
$[G,G]$ are again irreducible, and all of them have the centre in their kernel,
so we reach the desired conclusion for the simple group $[G/Z(G),G/Z(G)]$ as
well.
\end{proof}

We have now shown the assertion of Theorem~1 for groups of classical Lie type:

\begin{cor}   \label{cor:class}
 Let $B$ be an $\ell$-block of a quasi-simple group of classical Lie type,
 with $\ell\ge5$. Then $B$ is not a minimal counterexample to~$\Ca$ or~$\Cb$.
\end{cor}

\begin{proof}
By Corollary~\ref{cor:def char} there are no counterexamples when $\ell$ is
the defining characteristic.
By Lemma~\ref{lem:BDR} and Proposition~\ref{prop:Enguehard} the non-unipotent
blocks do not give rise to minimal counterexamples of either~$\Ca$ or~$\Cb$. By
Theorems~\ref{thm:SLn} and~\ref{thm:class} neither do the unipotent blocks.
\end{proof}

\section{Groups of exceptional Lie type in non-defining characteristic}   \label{sec:cross exc}

To complete the proof of Theorem~1 we need to deal with the blocks of
exceptional groups of Lie type in non-defining characteristic~$\ell$.

\subsection{Exceptional groups of small rank}

We first consider the five series of exceptional groups of small rank. The
following result has been communicated to us by Frank Himstedt as a
consequence of his investigation of blocks of the Steinberg triality groups
\cite{H07,HH13}:

\begin{prop}   \label{prop:k(B) 3D4}
 Let $B$ denote the principal $\ell$-block of $\tw3D_4(q)$. Then:
 \begin{enumerate}
  \item[\rm(a)] $k(B)=9+2^{a+1}+(4^{a-1}-1)/3$ for $\ell=2$,
  \item[\rm(b)] $k(B)=7+3^{a+1}+(9^a-1)/4$ for $\ell=3$.
 \end{enumerate}
\end{prop}

(We remark that these values are not given correctly in \cite[p.~68]{DM87}.)

\begin{prop}   \label{prop:exc small}
 Let $G$ be one of $\tw2B_2(q^2)$, $^2G_2(q^2)$, $\tw2F_4(q^2)'$, $G_2(q)$ or
 $\tw3D_4(q)$. Then all $\ell$-blocks of $G$ satisfy inequalities~$\Ca$
 and~$\Cb$ for all primes~$\ell$.
\end{prop}

\begin{proof}
By Theorem~\ref{thm:def char} we may assume that $\ell$ does not divide $q^2$.
By Theorem~\ref{thm:abelian}, we can furthermore discard the cases when Sylow
$\ell$-subgroups of $G$ are abelian. Hence by \cite[Thm.~25.14]{MT}, the prime
$\ell$ has to divide the order of the Weyl group of $\bG$.  \par
For the Suzuki groups $\tw2B_2(q^2)$ as well as for the Ree groups
$^2G_2(q^2)$ all Sylow $\ell$-subgroups for non-defining primes $\ell$
are abelian. The group $\tw2F_4(2)'$ was already treated in
Proposition~\ref{prop:spor}. For the Ree groups $G=\tw2F_4(q^2)$,
$q^2=2^{2f+1}$ with $f\ge1$, by \cite[Bem.~2]{Ma2F4}, only the principal
3-block $B$ has non-abelian defect groups. Here,
$k(B)=(3^{2a}+36\cdot3^a+555)/48$, where $3^a=(q^2+1)_3$, while a Sylow
3-subgroup $D$ of $G$ is an extension of a homocyclic group
$C_{3^a}\times C_{3^a}$ with the cyclic group of order~3. In particular
$k(D)\ge 3^{2a-1}$, and $\Cb$ follows when $a>1$. For $a=1$ we have $k(B)=14$,
$k(D)=11$ and $l(B)\ge2$, so again we are done. It is easy to see from the
character degrees given in \cite{Ma2F4} that furthermore $k_0(B)=9$, and as
we have $k(D')=3^{2a-1}$ this also shows~$\Ca$.
\par
The only relevant primes for the groups $G=G_2(q)$  are $\ell=2,3$. The blocks
and their invariants have been determined by Hiss and Shamash \cite{HS90,HS92}.
When $\ell=2$, the principal 2-block has $k(B)=9+2^{a+1}+(4^{a-1}-1)/3$,
$k_0(B)=8$ and $l(B)=7$, where $2^a$ is the precise power of 2 dividing
$q-\eps$, with $\eps\in\{\pm1\}$ such that $q\equiv\eps\pmod4$. A Sylow 
2-subgroup $D$ of $G$ is contained
in the normaliser of a maximally split torus or of a Sylow 2-torus and thus
is an extension of a homocyclic group $C_{2^a}\times C_{2^a}$ with a Klein
four group. So $k(D)\ge 2^{2a-2}$. Inequality~$\Cb$ follows by using that
necessarily $2^a\ge4$. Moreover, we have $D'$ is abelian of order~$2^{2a-1}$,
so $\Ca$ is also satisfied.
\par
For $q\equiv\pm1\pmod{12}$ there is a further 2-block $B$ in the Lusztig series
of an isolated element $s$ of order~3
whose defect groups are isomorphic to
Sylow 2-subgroups of $C_G(s)\cong\SL_3(q)$ or $\SU_3(q)$ respectively
(see \cite{HS92}). Thus they are central products of a semidihedral group with
a cyclic group and thus covered by Proposition~\ref{prop:Sambale}. All other
blocks have abelian or semidihedral defect groups.
 \par
When $\ell=3$ only the principal 3-block $B$ of $G_2(q)$, $3{\not|}q$,
has non-abelian defect
groups, with $k(B)=8+2\cdot 3^a+(3^a-3)^2/12$, $k_0(B)=9$ and $l(B)=7$, where
$3^a$ is the precise power of~3 dividing $q-\eps$, with $q\equiv\eps\pmod3$.
A Sylow 3-subgroup $D$ of $G_2(q)$ is contained in a subgroup $\SL_3(\eps q)$.
Thus, it is an extension of a homocyclic group $C_{3^a}\times C_{3^a}$ with the
cyclic group of order~3. So clearly $k(D)\ge 3^{2a-1}$ and the derived
subgroup is abelian of index~9, so $k(D')= 3^{2a-1}$. The inequalities follow
from this.
\par
The blocks of groups $\tw3D_4(q)$ were determined by Deriziotis--Michler
\cite{DM87}. In particular, the defect groups of any 2-block of non-maximal
defect are either abelian, semidihedral or semidihedral central product with
abelian, so by Proposition~\ref{prop:Sambale} we need not consider these
further. Any Sylow 2-subgroup $D$ of $\tw3D_4(q)$ is contained in a subgroup
$G_2(q)$ and thus, as we saw above, $k(D)\ge 2^{2a-2}$ and $k(D')=2^{2a-1}$,
with $2^a=(q-\eps)_2$ where $q\equiv\eps\pmod4$. Moreover, the invariants
$k(B)$ (given in Proposition~\ref{prop:k(B) 3D4}), $k_0(B)=8$ and $l(B)=7$ are
the same as for the principal 2-block of $G_2(q)$. The inequalities are thus
satisfied for $a\ge2$ (which by the definition of $\eps$ is always the case
here).

For $\ell=3$ any block with non-abelian defect has the Sylow 3-subgroups
as defect groups by \cite[Prop.~5.4]{DM87}. The only 3-block of maximal defect
is in fact the principal block, with $k(B)=7+3^{a+1}+(9^a-1)/4$ (see
Proposition~\ref{prop:k(B) 3D4}), $k_0(B)=9$ and $l(B)=7$.
Depending on the congruence of $q$ modulo~3, a Sylow 3-subgroup $D$ of $G$ is
contained in the normaliser of a maximally split torus of $G$ or of its Ennola
dual. Hence it is an extension of an abelian group $C_{3^a}\times C_{3^{a+1}}$
by the cyclic group of order~3, and $k(D)\ge 3^{2a}$. It contains a Sylow
3-subgroup of $G_2(q)$, from which we deduce that $k(D')\ge 3^{2a-1}$. The
claim follows.
\end{proof}

\subsection{Exceptional groups of large rank}

To deal with the exceptional groups of large rank, we need some preparations.
We continue to use the setup from Section~\ref{sec:def char} with $\bG$ a
simply connected simple algebraic group with Frobenius map $F$.

\begin{lem}   \label{lem:bound unip}
 Let $\bG$ be simple of one of the types given in Table~\ref{tab:nr unip}. Then
 the number of unipotent characters $|\cE(\bL^F,1)|$ for any proper $F$-stable
 Levi subgroup $\bL<\bG$ is bounded above as shown.
\end{lem}

\begin{table}[htb]
\caption{Numbers of unipotent characters and upper bounds}  \label{tab:nr unip}
$$\begin{array}{c|cccccccccccc}
 \quad G& A_1& A_2& B_2& A_3& B_3& A_4& B_4& D_4& \tw2D_4& F_4\\
\noalign{\hrule}
   |\cE(G,1)|& 2& 3& 6& 5& 12& 7& 25& 14& 10& 37\\
 |\cE(L,1)|\le& 0& 2& 2& 3&  6& 5& 12&  8&  5& 12\\
\end{array}$$
$$\begin{array}{c|ccccccccccc}
 \quad G& A_5& \tw{(2)}D_5& A_6& D_6& \tw2D_6& \tw{(2)}E_6& A_7& \tw{(2)}D_7& E_7& E_8\\
\noalign{\hrule}
   |\cE(G,1)|& 11& 20& 15& 42& 36& 30& 22& 65& 76& 166\\
 |\cE(L,1)|\le& 9& 14& 15& 28& 20& 20& 25& 42& 42& 76\\
\end{array}$$
\end{table}

\begin{proof}
By Lusztig's results the number of unipotent characters of the $F$-fixed points
of a connected reductive group $\bH$ only depends on the root system of $\bH$
and the action of $F$ on it (see e.g.~\cite{BMM}). Thus, inductively we are
done if we can show that
for any indecomposable parabolic subsystem $\Phi$ of the root system of $\bG$,
and for every maximal parabolic subsystem $\Psi$ of $\Phi$ the number of
unipotent characters of any connected reductive group with root system $\Psi$
is smaller than the stated bound. Now the maximal parabolic subsystems of
$\Phi$ are obtained by removing one node in the Dynkin diagram of $\Phi$.
From this together with the list of numbers of unipotent characters reproduced
in Table~\ref{tab:nr unip} it is now straightforward to conclude. As an
example, when $\Phi$ has type $F_4$, the maximal parabolic subsystems are of
types $B_3$,$A_2A_1$ and $C_3$, with $12,6,12$ unipotent characters
respectively, consistent with the claimed bound.
\end{proof}

\begin{prop}   \label{prop:unip exc}
 Let $B$ be a unipotent $\ell$-block of a quasi-simple group of exceptional
 Lie type, with $\ell\ge5$. Then $B$ is neither a minimal counterexample to~$\Ca$
 nor to~$\Cb$.
\end{prop}

\begin{proof}
Let $G$ be a quasi-simple group of exceptional Lie type. By
Theorem~\ref{thm:def char} we may assume that $\ell$ is not the defining
characteristic of $G$. By Proposition~\ref{prop:exc small} we may also assume
that $G$ is of type $F_4$, $E_6$, $\tw2E_6$, $E_7$ or $E_8$.
Furthermore, by Proposition~\ref{prop:exc cover}, $G$ is not an exceptional
covering group, so we may further assume that $G$ is a central quotient of a
group $\bG^F$ as above.  \par
If Sylow $\ell$-subgroups of $G$ are non-abelian, then $\ell$ divides the order
of the Weyl group of $\bG$ (see \cite[Thm.~25.14]{MT}). In particular
$\ell\le7$. Now it is easily seen
that Sylow $7$-subgroups can be nonabelian only when $G=E_7(q)$ or $E_8(q)$ and
$7|(q^2-1)$, and Sylow 5-subgroups are only non-abelian when $G$ is of type
$E$, and moreover $5|(q^2-1)$ when $G\ne E_8(q)$. The unipotent $\ell$-blocks
have been classified in \cite{CE94} for good primes, and for bad primes in
\cite{En00}. It ensues that for $\ell\ge5$ these are in bijection with
$d$-Harish-Chandra series of unipotent characters of $G$, where $d$ is the
order of $q$ modulo~$\ell$. But then from the explicit knowledge of these
series, \cite[Thm. (ii)]{CE94} shows that defect groups of non-principal
unipotent $\ell$-blocks are always abelian in our cases.  \par
Hence, we only need to consider the principal $\ell$-block $B_0$. Here
$$\Irr(B_0)\subseteq\cE_\ell(G,1)=\coprod_t\cE(G,t)$$
where the union runs over $\ell$-elements $t\in G^*$ up to conjugation. In
particular the union certainly has at most $k(D)$ terms, where $D$ is a Sylow
$\ell$-subgroup of $G$. In order to prove our inequalities, we will do two
things: first bound $|\cE(G,t)|$ suitably, and secondly relate the number of
conjugacy classes of $\ell$-elements in $G$ to $k(D)$.
\par
For the first step, observe that by Lusztig's Jordan decomposition $\cE(G,t)$
is in bijection with $\cE(C_{G^*}(t),1)$, hence we need to control the number
of unipotent characters of the centralisers of $\ell$-elements in $G$. The
candidates for these centralisers can easily be enumerated by the algorithm of
Borel--de Siebenthal from the extended Dynkin diagram \cite[Thm.~B.18]{MT}
(note that in types $E_n$ all maximal subsystems are necessarily closed).
Since $\ell$ is prime to $|Z(\bG)|$ all such
centralisers are connected. Moreover, unless $G$ is of type $E_8$ and $\ell=5$,
the prime $\ell$ is good for $G$ and thus the centraliser of any non-trivial
$\ell$-element of $G^*$ is even a Levi subgroup. Here, an upper bound for
$|\cE(C_{G^*}(t),1)|$ is given in Table~\ref{tab:nr unip}. When $\ell=5$
in $G=E_8(q)$, the only isolated centraliser of a 5-element is of type
$A_4A_4$, which has at most~49 unipotent characters (depending on its rational
type). Note that when $q^2\equiv4\pmod5$ then there is no 5-element in $E_8(q)$
with centraliser of type $E_7$, $D_6$ or $D_7$, so in that case the number of
unipotent characters of any proper centraliser is bounded above by~36.  \par
For the comparison of conjugacy classes in $G$ and in $D$ let's start with the
case of $G=E_6(q)$ with $\ell=5|(q-1)$. Then a Sylow 5-subgroup $D$ of $G$ is
contained in the normaliser of a maximally split torus $T$ of $G$. If
$1\ne t\in D\cap T$ then $t$ is conjugate to at least 27 elements of
$D$ (the index of the largest proper subgroup of the Weyl group $W$ of $G$),
while in $D$ it is conjugate to at most 5 elements. So the number of conjugacy
classes of such elements in $G$ is at most one fifth of their number in $D$.
On the other hand, if $t\in D\setminus T$ then it centralises a Sylow 5-torus
of $G$, hence its centraliser lies in a subgroup of type $A_1(q).(q^5-1)$. In
particular $|\cE(C_{G^*}(t),1)|\le2$. Taking together our results we find that
$k(B_0)\le c k(D)$ with $c=30/5=6$.  

\begin{table}[htb]
\caption{$l(B_0)$, $c$ and $k_0(B_0)$ in exceptional types}  \label{tab:l(B_0)}
$$\begin{array}{c|ccccccc}
    G& ^{\eps}\!E_6(q)& E_7(q)& E_7(q)& E_8(q)& E_8(q)& E_8(q)\\
 \ell& q\equiv\eps1\,(5)& q^2\equiv1\,(5)& q^2\equiv1\,(7)&
     q^2\equiv1\,(5)& q^2\equiv4\,(5)& q^2\equiv1\,(7)\\
\noalign{\hrule}
   l(B_0)& 25& 60& 60& 112& 59& 112\\
        c&  6&  1&  15/4&  25/4&  3/4&  38/17\\
 k_0(B_0)\ge& 5+5^{a+1}& 14& & 40& 20& \\
\noalign{\hrule}
\end{array}$$
\end{table}

With exactly the same arguments we obtain the constants $c$ listed in
Table~\ref{tab:l(B_0)} in all cases except for $E_7(q)$ with $\ell=5$. In the
latter case we need to be a bit more careful. The centre of $D$ has order
$5^{3a}$, all other elements in $D$ lie in orbits of length at least 2520
under the action of $W$, and their centraliser in $G$ is of type $A_3+A_3+A_1$
or smaller, so a corresponding Lusztig series has at most 50 elements. Thus
we obtain $k(B)\le 65\cdot 5^{3a}+50(5^{7a+1}-5^{3a})/2520$. With
$k(D)=5^{2a}((5^{5a}-5^a)/5+5^{a+1})$ this yields the bound $c=1$.
\par
The number $l(B_0)$ is given in Table~\ref{tab:l(B_0)}. (When $\ell$ is good
for $G$ then by \cite{GH91} this is just $|\cE(G,1)\cap\Irr(B_0)|$.) As visibly
$c\le l(B_0)$ in all cases we have shown~$\Cb$.
\par
We now turn to~$\Ca$. Here it suffices to see that
$$c\,k(D) \le k_0(B_0)\,k(D').$$
First consider $G=E_6(q)$, with $\ell=5$ dividing $q-1$. A Sylow 5-subgroup
$D$ of $G$ is contained in a Levi subgroup of type $A_4$ and thus isomorphic to
a Sylow 5-subgroup of $\GL_5(q).(q-1)$. Thus $k(D)=5^a((5^{5a}-5^a)/5+5^{a+1})$
and $k(D')=5^{4a}$, where $5^a$ is the precise power of~5 dividing $q-1$ (see
Lemma~\ref{lem:k(D)}).
Furthermore, there are 10 unipotent characters of height~0 in $B_0$. The
centre of a Levi subgroup of $G^*$ of type $A_4$ contains an abelian subgroup
$A$ of order $5^{2a}$ any element of which is $G^*$-conjugate to at most one
further element of $A$. The corresponding Lusztig series each contain 5
characters in $B_0$ of height~0, so $k_0(B_0)\ge 10+5(5^{2a}-1)/2$. We obtain
that $k_0(B_0)\,k(D')/k(D)\ge 25/2>c=6$.
Exactly the same arguments apply for $G=\tw2E_6(q)$ with $5|(q+1)$.
\par
For $G=E_7(q)$ with $\ell=5$ dividing $q-1$ a Sylow $5$-subgroup $D$
of $G$ is again contained in a Levi subgroup of type $A_4$, so is isomorphic
to a Sylow $5$-subgroup of $\GL_5(q).(q-1)^2$. Hence we have
$k(D)=5^{2a}((5^{5a}-5^a)/5+5^{a+1})$ and $k(D')=5^{4a}$. Besides the 30
unipotent characters of height~0 in $B_0$ there are at least $5(5^{3a}-1)/12$
further such characters in the Lusztig series of 5-elements in the centre of
a Levi subgroup of $G^*$ of type $A_4$. Thus $k_0(B_0)\,k(D')/k(D)\ge25/12>c=1$.
The case when $5|(q+1)$ is completely analogous.
\par
For $G=E_7(q)$ with $\ell=7$ dividing $q-1$ a Sylow $7$-subgroup $D$ of $G$ is
contained in a Levi subgroup of type $A_6$, so isomorphic to a Sylow
$7$-subgroup of $\GL_7(q)$. Hence we have $k(D)=(7^{7a}-7^a)/7+7^{a+1}$ and
$k(D')=7^{6a}$. Besides the 14 unipotent characters of height~0 in $B_0$ there
are further $7(7^a-1)/2$ such characters in the $(7^a-1)/2$ Lusztig series
of 7-elements in $G^*$ with centraliser containing $A_6(q)$. Again
$k_0(B_0)\,k(D')/k(D)\ge 24>c=15/4$.
\par
For $E_8(q)$ with $\ell=5|(q-1)$, a Sylow 5-subgroup $D$ of $G$ is contained in
a maximal rank subgroup of type $A_4(q)^2$, so it is a homocyclic group of
order $5^{8a}$ extended by an elementary abelian group of order~25,
where $5^a$ is the precise power of 5 dividing $q-1$. Here
$k(D)=(5^{4a-1}+24)^2$ and $k(D')=5^{8a-2}$, and furthermore $k_0(B_0)=40$.
So we have $k_0(B_0)\,k(D')/k(D)>28>c$.

For $E_8(q)$ with $\ell=5|(q^2+1)$, a Sylow 5-subgroup $D$ of $G$ is contained
in a maximal rank subgroup $\tw2A_4(q^2)$, hence isomorphic to a Sylow
5-subgroup of $\SU_5(q^2)$. So we have $k(D)=5^{4a-1}+24$ and
$k(D')=5^{4a-1}$, where $5^a=(q^2+1)_5$. Since $k_0(B_0)=20$ we get
$k_0(B_0)\,k(D')/k(D)>16>c$.

For $G=E_8(q)$ with $\ell=7$ dividing $q-1$ a Sylow $7$-subgroup $D$ of $G$ is
contained inside a Levi subgroup of type $A_6$, so isomorphic to a Sylow
$7$-subgroup of $\GL_7(q).(q-1)$. Hence we have
$k(D)=7^a((7^{7a}-7^a)/7+7^{a+1})$ and $k(D')=7^{6a}$. Besides the 28 unipotent
characters of height~0 in $B_0$ there are 7 height zero characters in at least
$(7^{2a}-1)/4$ further Lusztig series of 7-elements in $G^*$ in the centre of
a Levi subgroup of type $A_6$. Thus $k_0(B_0)\,k(D')/k(D)>12>c$, which
completes the proof.
\end{proof}

\begin{thm}   \label{thm:exc}
 Let $B$ be an $\ell$-block of a quasi-simple group of exceptional Lie type,
 with $\ell\ge5$. Then $B$ is not a minimal counterexample to~$\Ca$ or~$\Cb$.
\end{thm}

\begin{proof}
By Lemma~\ref{lem:BDR} we only need to consider isolated $\ell$-blocks $B$.
The unipotent blocks have been dealt with in Proposition~\ref{prop:unip exc}.
We are left with blocks labelled by isolated $\ell'$-elements $1\ne s\in G^*$.
By Proposition~\ref{prop:Enguehard}, $\ell$ is bad for $G$, so $G=E_8(q)$ and
$\ell=5$. The isolated 5-blocks with non-abelian defect were determined in
\cite[Prop.~6.10]{KM13}: they are those corresponding to semisimple
$5'$-elements $s$ with centraliser of type
$$A_8,\ A_7+A_1,\ A_5+A_2+A_1,\ D_8,\ E_7+A_1,\ D_5+A_3,\ E_6+A_2,$$
when $q\equiv1\pmod5$, their Ennola duals when $q\equiv-1\pmod5$, while there
are none when $q^2\equiv-1\pmod5$. In the first three cases, the corresponding
Brou\'e--Michel union $\cE_5(G,s)$ is a single 5-block $B$, and $l(B)$ agrees
with the corresponding number in the principal 5-block of $C_{G^*}(s)^*$, so we
are done by Lemma~\ref{lem:one block}.
\par
Next consider the case of centraliser $D_8(q)$. Here, $\cE_5(G,s)$ is the
union of two blocks, one with trivial Harish-Chandra source, the other above
a cuspidal character of $D_4(q)$. The second block has abelian defect,
and for the first we have
$$k(B)\le|\cE_5(G,s)|
  =k(5,a,1,8)+k(5,a,1,4)\le 2k(5^a,3)\cdot 5^{5a-2.7}+5^{4a}$$
by Lemmas~\ref{lem:k(B)} and~\ref{lem:k(a,b)}, while defect groups are of
the form $C_{5^a}^3\times C_{5^a}\wr C_5$, so
$$k(D)\ge 5^{8a-1},\quad k(D')\ge 5^{4a},\quad\text{and }
  k_0(B)\ge k(5^a,3)k(5^{a+1},1)=5^{a+1}k(5^a,3),$$
from which we may conclude.
\par
The remaining three cases are settled analogously.
\end{proof}


\end{document}